\documentclass[12pt]{amsart}

\usepackage{amsmath}
\usepackage{amsfonts}
\usepackage{mathrsfs}
\usepackage{amsthm}
\usepackage[colorlinks=true]{hyperref}
\usepackage{epsfig}
\usepackage[dvipsnames]{xcolor}

\newtheorem{theorem}{Theorem}[section]
\newtheorem{proposition}[theorem]{Proposition}
\newtheorem{lemma}[theorem]{Lemma}
\newtheorem{corollary}[theorem]{Corollary}
\theoremstyle{remark}
\newtheorem*{claim}{Claim}
\newtheorem*{obs}{Observation}
\newtheorem{example}{Example}

\newcommand{\mc}[0]{\mathcal}

\newcommand{\bb}[0]{\mathbb}
\newcommand{\bs}[0]{\backslash}
\newcommand{\GB}{(G,\mathcal B)}
\newcommand{\Bb}{\mathcal B}
\newcommand{\Uu}{\mathcal U}

\newcommand{\del}{\backslash}
\DeclareMathOperator{\cl}{cl}
\DeclareMathOperator{\loops}{loops}
\newcommand{\Ii}{\mathcal I}
\newcommand{\Ff}{\mathcal F}
\newcommand{\Ll}{\mathcal L}
\newcommand{\Qq}{\mathcal Q}

\newcommand{\Cc}{\mathcal C}
\newcommand{\Tt}{\mathcal T}

\newcommand{\FF}{\mathbb F}
\newcommand{\comout}[1]{}
\newcommand{\br}[1]{\left( #1 \right)}
\newcommand{\iso}{\cong}
\newcommand{\QGone}{\hyperlink{(qg4)}{{\scshape (q1)}}} 
\newcommand{\QGtwo}{\hyperlink{(qg4)}{{\scshape (q2)}}} 
\newcommand{\QGthree}{\hyperlink{(qg4)}{{\scshape (q3)}}} 
\newcommand{\QGfour}{\hyperlink{(qg4)}{{\scshape (q4)}}} 
\newcommand{\BGone}{\hyperlink{BGM(1)}{{\scshape (b1)}}} 
\newcommand{\BGtwo}{\hyperlink{BGM(2)}{{\scshape (b2)}}} 
\newcommand{\BGthree}{\hyperlink{BGM(3)}{{\scshape (b3)}}} 

\begin{document}

\title{Describing Quasi-Graphic Matroids} 

\author[Bowler]{Nathan Bowler$^*$} 
\address{${}^*$Department of Mathematics, University of Hamburg, Germany} 
\email{nathan.bowler@uni-hamburg.de} 

\author[Funk]{Daryl Funk${}^\dagger$}
\address{${}^\dagger$Department of Mathematics, Douglas College, British Columbia, Canada} 
\email{funkd@douglascollege.ca} 

\author[Slilaty]{Daniel Slilaty${}^\ddagger$}
\address{${}^\ddagger$Department of Mathematics and Statistics, Wright State University, Dayton, OH 45435. Work partially supported by a grant from the Simons Foundation \#246380.}
\email{daniel.slilaty@wright.edu}

\date{May 26, 2025} 

\maketitle

{\color{black}
This is a revised version of our original paper, \emph{Describing Quasi-graphic Matroids} \cite{MR4037634}, incorporating the corrections published in \cite{QG_correction}. 
Our main theorem as stated in  \cite{MR4037634} is missing the required assumption that matroids should be connected; \cite{QG_correction} addresses this oversight. 
Those unfamiliar with the original paper will find in this version a complete, correct description of quasi-graphic matroids, sparing them the inconvenience of having to read both the original paper and the corrections presented in \cite{QG_correction}. 

We also present here some new results 
that do not appear in \cite{MR4037634} or \cite{QG_correction}. 
These appear in Section \ref{sec:Biased_graphic_matroids}. 
Of particular interest to readers already familiar with \cite{MR4037634} and \cite{QG_correction} may be the following: 
\begin{itemize}
\item 
Given a matroid $M$ and a graph $G$, of the four axioms for quasi-graphic matroids, three may be checked in time polynomial in $|E(M)|$, but the fourth in general cannot. Jim Geelen has asked for a condition, checkable in polynomial time, that could replace this fourth axiom [personal communication]. 
We provide such an alternative (Theorem \ref{Felix}). 
\item A construction extending the bracelet function and tripartion constructions to direct sums of matroids (Theorem \ref{thm:component-wise-proper_bracelet_function_defn}). 
\end{itemize} 
}

% \newpage 
% \mbox{}
% \vfill 

\begin{abstract}
The class of quasi-graphic matroids recently introduced by Geelen, Gerards, and Whittle generalises each of the classes of frame matroids and lifted-graphic matroids introduced earlier by Zaslavsky. 
For each biased graph $\GB$ Zaslavsky defined a unique lift matroid $L\GB$ and a unique frame matroid $F\GB$, each on ground set $E(G)$. 
We show that in general there may be many quasi-graphic matroids on $E(G)$ and describe them all: 
for each graph $G$ and partition $(\Bb,\Ll,\Ff)$ of its cycles such that $\Bb$ satisfies the theta property and each cycle in $\Ll$ meets each cycle in $\Ff$, there is a quasi-graphic matroid $M(G,\Bb,\Ll,\Ff)$ on $E(G)$. 
Moreover, every quasi-graphic matroid arises in this way. 
We provide cryptomorphic descriptions in terms of subgraphs corresponding to circuits, cocircuits, independent sets, and bases. 
Equipped with these descriptions, we prove some results about quasi-graphic matroids. 
In particular, we provide alternate proofs that do not require 3-connectivity of two results of Geelen, Gerards, and Whittle for 3-connected matroids from their introductory paper: namely, 
that every quasi-graphic matroid linearly representable over a field is either lifted-graphic or frame,  
and that if a matroid $M$ has a framework with a loop that is not a loop of $M$ then $M$ is either lifted-graphic or frame. 
We also provide sufficient conditions for a quasi-graphic matroid to have a unique framework. 

Zaslavsky has asked for those matroids whose independent sets are contained in the collection of independent sets of $F\GB$ while containing those of $L\GB$, for some biased graph $\GB$. 
Adding a natural (and necessary) non-degeneracy condition defines a class of matroids, which we call \emph{biased-graphic}. 
We show that the class of biased-graphic matroids almost coincides with the class of quasi-graphic matroids: every quasi-graphic matroid is biased-graphic, and if $M$ is a biased-graphic matroid that is not quasi-graphic then $M$ is a 2-sum of a frame matroid with one or more lifted-graphic matroids. 
\end{abstract}

% \newpage 

{
  \hypersetup{hidelinks}
  \tableofcontents
}

% \newpage 

\section{Context and motivation} 

In series of foundational papers
\cite{Zaslavsky:BG1, Zaslavsky:BG2, MR1273951, MR1328292, MR2017726}
Thomas Zaslavsky introduced \emph{biased graphs}, their associated \emph{frame} and \emph{lift} matroids, and established their basic properties.
A matroid is a \emph{frame matroid} if it may be extended so that it has a basis $B$ such that every element is spanned by at most two elements of $B$.
Such a basis is a \emph{frame} for the matroid.
A matroid $M$ is a \emph{lift} (or \emph{lifted-graphic}) \emph{matroid} if it is an elementary lift of a graphic matroid; that is, if there is a matroid $N$ with $E(N) = E(M) \cup \{e\}$ such that $N \del e = M$ and $N/e$ is graphic.

These are fundamental and important classes of matroids.
Frame matroids were introduced by Zaslavsky as a significant generalisation of Dowling geometries 
(the cycle matroid of a complete graph is a Dowling geometry over the trivial group). 
Moreover, classes of representable frame matroids play an important role in the matroid minors project of Geelen, Gerards, and Whittle \cite[Theorem
3.1]{GeelenGerardsWhittle:HighlyConnected}, analogous to that of graphs embedded on surfaces in graph structure theory. 

Geelen, Gerards, and Whittle recently introduced the class of {quasi-graphic} matroids as a common generalisation of each of these classes  \cite{JGT:JGT22177}. 
A matroid $M$ is \emph{quasi-graphic} if it has a \emph{framework}: 
that is, a graph $G$ with (i) $E(G) = E(M)$, such that (ii) the rank of the edge set of each component of $G$ is at most the size of its vertex set, 
(iii) for each vertex $v \in V(G)$ the closure in $M$ of $E(G-v)$ does not contains an edge with endpoints $v, w$ with $w \neq v$, 
and (iv) no circuit of $M$ induces a subgraph in $G$ of more than two components. 
For each quasi-graphic matroid $M$ there is a biased graph $\GB$, where $G$ is a framework for $M$ and $\Bb$ is the set of cycles of $G$ that are circuits of $M$. 
However, there is not enough information provided by a biased graph to determine a quasi-graphic matroid. 
We define two notions that provide the missing information: \emph{bracelet functions} and \emph{proper triparitions}. 
A \emph{bracelet} is a vertex disjoint pair of unbalanced cycles in a biased graph. 
A \emph{bracelet function} $\chi$ is a function that maps each bracelet of a biased graph to $\{\mathsf{dependent}, \mathsf{independent}\}$. 
If $\chi$ obeys a certain condition then we say $\chi$ is \emph{proper} and 
we may define a matroid $M(G,\Bb,\chi)$ on $E(G)$, in which the independence of each bracelet is given by $\chi$. 
Given a graph $G$, a \emph{proper tripartition} of the cycles of $G$ is a partition $(\Bb,\Ll,\Ff)$ of its collection of cycles such that $\Bb$ obeys the theta property and every cycle in $\Ll$ meets every cycle in $\Ff$. 
Thus in a proper tripartition every bracelet either has both its cycles in $\Ll$ or both its cycles in $\Ff$. 
Given a graph $G$ together with a proper tripartition $(\Bb,\Ll,\Ff)$ of its cycles, we may define a matroid $M(G,\Bb,\Ll,\Ff)$ on $E(G)$, in which a bracelet is independent precisely when both of its cycles are in $\Ff$. 
We can now state our main result. 

\comout{OLD:
\begin{theorem} \label{thm:equivalences}
Let $M$ be a matroid and let $(G, \Bb)$ be a biased graph with $E(G)=E(M)$. 
The following are equivalent. 
\begin{enumerate} 
\item There is a proper bracelet function $\chi$ for $G$ such that $M=M(G,\Bb,\chi)$. 
\item There is a proper tripartition $(\Bb,\Ll,\Ff)$ of the cycles of $G$ such that $M=M(G,\Bb,\Ll,\Ff)$. 
\item $M$ is quasi-graphic with framework $G$ and $\Bb$ is the set of cycles of $G$ that are circuits of $M$. 
\end{enumerate} 
\end{theorem}
}

{\color{black}
\begin{theorem} \label{fixed_thm:equivalences}
Let $M$ be a connected matroid and let $(G,\Bb)$ be a biased graph with $E(G) = E(M)$. 
The following are equivalent. 
\begin{enumerate}%[label=\textup{(\arabic*)}]
\item There is a proper bracelet function $\chi$ for $G$ such that $M=M(G,\Bb,\chi)$. 
\item There is a proper tripartition $(\Bb,\Ll,\Ff)$ of the cycles of $G$ such that $M=M(G,\Bb,\Ll,\Ff)$. 
\item $M$ is quasi-graphic with framework $G$ and $\Bb$ is the set of cycles of $G$ that are circuits of $M$. 
\end{enumerate} 
\end{theorem}

{\color{black}More precisely, the statement we in fact prove is the following.} 

\begin{theorem} \label{fixed_2_thm:equivalences}
Let $M$ be a matroid and let $(G,\Bb)$ be a biased graph with $E(G) = E(M)$. 
If either $M$ is connected or $G$ is connected, then the following are equivalent. 
\begin{enumerate}%[label=\textup{(\arabic*)}]
\item There is a proper bracelet function $\chi$ for $G$ such that $M=M(G,\Bb,\chi)$. 
\item There is a proper tripartition $(\Bb,\Ll,\Ff)$ of the cycles of $G$ such that $M=M(G,\Bb,\Ll,\Ff)$. 
\item $M$ is quasi-graphic with framework $G$ and $\Bb$ is the set of cycles of $G$ that are circuits of $M$. 
\end{enumerate} 
\end{theorem}
}

\subsection{Frame matroids}
In \cite{MR1273951} Zaslavsky showed that the class of frame matroids is precisely that of matroids arising from \emph{biased graphs}, as follows.
Let $M$ be a frame matroid on ground set $E$, with frame $B$.
By adding elements in parallel if necessary, we may assume $B \cap E = \emptyset$.
Hence for some matroid $N$, $M = N \del B$ where $B$ is a basis for $N$ and every element $e \in E$ is minimally spanned by either a single element or a pair of elements in $B$.
Let $G$ be the graph with vertex set $B$ and edge set $E$, in which $e$ is a loop with endpoint $f$ if $e$ is parallel with $f \in B$, and otherwise $e$ is an edge with endpoints $f, f' \in B$ if $e \in \cl\{f,f'\}$.
The edge set of a cycle of $G$ is either independent or a circuit in $M$.
A cycle $C$ in $G$ whose edge set is a circuit of $M$ is said to be \emph{balanced}; otherwise $C$ is \emph{unbalanced}.
Thus the cycles of $G$ are partitioned into two sets: those that are circuits and those that are independent.
The collection of balanced cycles of $G$ is denoted $\Bb$.
The \emph{bias} of a cycle is given by the set of the bipartition to which it belongs. 
Together the pair $(G,\Bb)$ is a \emph{biased graph}.
We say such a biased graph $\GB$ \emph{represents} the frame matroid $M$, and we write $M = F\GB$.

The circuits of a frame matroid $M$ may be precisely described in terms of biased subgraphs of such a biased graph $\GB$.
A \emph{theta graph} consists of a pair of distinct vertices with three internally disjoint paths between them.
The circuits of $M$ are precisely those sets of edges inducing one of:
a balanced cycle,
a theta subgraph in which all three cycles are unbalanced,
two edge-disjoint unbalanced cycles intersecting in exactly one vertex, or
two vertex-disjoint unbalanced cycles along with a minimal path connecting them.
The later two biased subgraphs are called \emph{handcuffs}, \emph{tight} or \emph{loose}, respectively.
It is a straightforward consequence of the circuit elimination axiom that a biased theta subgraph of $\GB$ may not contain exactly two balanced cycles.
We call this the \emph{theta property}.

Zaslavsky further showed \cite{MR1273951} that conversely, given any graph $G$ and partition $(\Bb,\Uu)$ of its cycles, all that is required for there to exist a frame matroid
whose circuits are given by the collection of biased subgraphs described above is that the collection $\Bb$ satisfy the theta property.

\subsection{Lifted-graphic matroids}
Let $N$ be a matroid on ground set $E \cup \{e\}$, and suppose $G$ is a graph with edge set $E$ and with cycle matroid $M(G)$ equal to $N/e$.
Then $M = N \del e$ is a lifted-graphic matroid.
Each cycle in $G$ is either a circuit of $N$, and so of $M$, or together with $e$ forms a circuit of $N$. 
Thus again the cycles of $G$ are naturally partitioned into two sets: those that are circuits of $M$ and those that are independent in $M$; thus a lifted-graphic matroid naturally gives rise to a biased graph.
Again, a cycle whose edge set is a circuit of $M$ is said to be balanced, and those whose edges form an independent set are unbalanced.
In \cite{Zaslavsky:BG2} Zaslavsky showed that the circuits of $M$ are precisely those sets of edges in $G$ inducing one of: a balanced cycle, a theta subgraph in which all three cycles are unbalanced, two edge disjoint unbalanced cycles meeting in exactly one vertex, or a pair of vertex-disjoint unbalanced cycles.
The later two biased subgraphs are called tight handcuffs and \emph{bracelets}, respectively.
Letting $\Bb$ denote the collection of balanced cycles of $G$, we again say the biased graph $\GB$ so obtained \emph{represents} the lifted-graphic matroid $M$ and write $M = L\GB$.
Just as with frame matroids, Zaslavsky showed that given any graph $G$ and partition $(\Bb,\Uu)$ of its cycles with $\Bb$ obeying the theta property, there is a lifted-graphic matroid $M=L\GB$ whose circuits are precisely those biased subgraphs described above, and that all lifted-graphic matroids arise from biased graphs in this way.

\subsection{Quasi-graphic matroids} \label{sec:quasigraphicintro}
In \cite{JGT:JGT22177}, Geelen, Gerards, and Whittle define the class of {quasi-graphic matroids}, as follows. 
For a vertex $v$, denote by $\loops(v)$ the set of loops incident to $v$.
Given a matroid $M$, a \emph{framework} for $M$ is a graph $G$ satisfying
\begin{enumerate}
\item[{\scshape (q1)}] \hypertarget{(qg1)} $E(G) = E(M)$,
\item[{\scshape (q2)}] \hypertarget{qg2} for each component $H$ of $G$, $r(E(H)) \leq |V(H)|$,
\item[{\scshape (q3)}] \hypertarget{(qg3)}for each vertex $v \in V(G)$, $\cl(E(G-v)) \subseteq E(G-v) \cup \loops(v)$, and
\item[{\scshape (q4)}] \hypertarget{(qg4)} if $C$ is a circuit of $M$, then the graph induced by $E(C)$ has at most two components.
\end{enumerate}
A matroid is \emph{quasi-graphic} if it has a framework. 
It is conjectured that, in contrast to the classes of lifted-graphic and frame matroids, the class of quasi-graphic matroids enjoys some nice properties. 
Chen and Geelen \cite{2017arXiv170304857C} recently showed that each of the classes of frame and lifted-graphic matroids have infinitely many excluded minors.
They conjecture that the class of quasi-graphic matroids has only finitely many excluded minors. 
And while Chen and Whittle \cite{ChenWhittle:HardToRecognize} have shown that there is no polynomial-time algorithm that can recognise, via a rank-oracle, whether a given matroid is a frame matroid or a lifted-graphic matroid, 
Geelen, Gerards, and Whittle conjecture \cite{JGT:JGT22177} that there is a such a polynomial-time algorithm for deciding whether or not a given 3-connected matroid is quasi-graphic. 

Let $M$ be a quasi-graphic matroid, and let $G$ be a framework for $M$. 
%Geelen, Gerards, and Whittle show that e
Every forest of $G$ is independent in $M$ \cite[Lemma 2.5]{JGT:JGT22177}, so every cycle of $G$ is either minimally dependent or independent. 
As before, let $(\Bb,\Uu)$ be the partition of the cycles of $G$ into two sets according to whether each cycle is a circuit ($\Bb$) or independent ($\Uu$) in $M$, and call those cycles in $\Bb$ \emph{balanced}.
Geelen, Gerards, and Whittle show that the collection $\Bb$ satisfies the theta property; 
the following lemma is an immediate consequnce of \cite[Lemma 3.3]{JGT:JGT22177} and axiom \QGfour. 

\begin{lemma} \label{lem:GGWquasicircuits}
Let $G$ be a framework for a matroid $M$. 
If $C$ is a circuit of $M$, then $C$ induces in $G$ one of: 
a balanced cycle, 
a theta with no cycle balanced, 
a tight handcuff, 
a loose handcuff, or 
a bracelet. 
\end{lemma}

Thus in a biased graph the edge set of a balanced cycle, a theta containing no balanced cycle, and tight handcuffs are circuits in each of a lifted-graphic, frame, and quasi-graphic matroid. 
The circuit-subgraphs of these matroids differ only in bracelets: in a lifted-graphic matroid all bracelets are dependent, in a frame matroid all bracelets are independent, while in general a quasi-graphic matroid has both dependent and independent bracelets.  
It follows that if $M$ is a quasi-graphic matroid with framework $G$, then setting $\Bb = \{ C : C$ is a cycle of $G$ and a circuit of $M\}$ yields a biased graph $\GB$ for which 
\[
\Ii(L\GB) \subseteq \Ii(M) \subseteq \Ii(F\GB).
\]

\subsection{Intermediate matroids} 
\label{sec:Intermediate_matroids}

In \cite{Zaslavsky:BG2} Zaslavsky asks for those matroids $M$ satisfying $\Ii\br{L\GB} \subseteq \Ii\br{M} \subseteq \Ii\br{F\GB}$ for some biased graph $\GB$, calling such a matroid $M$ \emph{intermediate} between $L\GB$ and $F\GB$.
Zaslavsky asks the following \cite[Problem 4.3]{Zaslavsky:BG2}.  
\begin{enumerate}
\item Given a biased graph $\GB$, what matroids $M$ on $E(G)$ may exist that are intermediate between $L\GB$ and $F\GB$?
\item Is there a systematic way to construct intermediate matroids?
\end{enumerate}
For a subset $S \subseteq E(G)$, denote by $G[S]$ the subgraph of $G$ induced by $S$ and by $\Bb_S$ the collection $\{ C \in \Bb : E(C) \subseteq S\}$.
Zaslavsky suggests that a ``systematic'' construction ought to be a mapping $\mathbf{M}$ from the set of biased graphs, or from some subset, to the set of matroids such that $\mathbf{M}\GB$ is a matroid on $E(G)$ and  for each $S \subseteq E(G)$, $\mathbf{M}(G[S], \Bb_S)$ is defined and equal to the restriction $\mathbf{M}\GB | S$ of $\mathbf{M}\GB$ to $S$.
The mappings $\mathbf{F}$ and $\mathbf{L}$ assigning to each biased graph $\GB$ its associated frame and lifted-graphic matroid, respectively, are such maps.
Zaslavsky asks for an \emph{intermediate-matroid construction}: that is, a map $\mathbf{M}$ respecting restriction such that 
\[
\Ii(L\GB) \subseteq \Ii(\mathbf{M}\GB) \subseteq \Ii(F\GB)
\]
for all biased graphs $\GB$.
Zaslavsky proves that the only intermediate-matroid constructions with domain all biased graphs are the mappings $\mathbf{F}$ and $\mathbf{L}$ \cite[Theorem 4.5]{Zaslavsky:BG2} and speculates that there exist intermediate constructions other than $\mathbf{F}$ and $\mathbf{L}$ with domain all biased graphs having no unbalanced loops \cite[page 66]{Zaslavsky:BG2}. 

It turns out that what is needed to answer Zaslavsky's questions (1) and (2) is a refinement of the notion of a biased graph, via a refinement of the partition $(\Bb,\Uu)$ of its cycles. 
If we wish to consider those matroids that are intermediate for some biased graph as a meaningful class of matroids, then it is also necessary to impose some kind of restriction on the relationships permitted between a graph and a matroid defined on its edge set. 
The following example may easily be generalised to show that all matroids are intermediate for some biased graph. 

\begin{example} \label{ex:bad_all_loops_example} 
Let $G$ be the biased graph consisting of $n$ vertices each incident to a single unbalanced loop. 
Then $L\GB \iso U_{1,n}$ and $F\GB \iso U_{n,n}$. 
Every loopless matroid on $n$ elements is intermediate between $L\GB$ and $F\GB$.
\end{example}

What are we to make of the problem illustrated by Example \ref{ex:bad_all_loops_example}, and how should we go about addressing it? 
The problem appears to be that components of the matroid do not correspond in any way with the components of the graph. 
Let us consider how matroid components and graph components align in the well-known classes of graphic, lifted-graphic, and frame matroids. 
Any condition we impose on our graphs and intermediate matroids defined on their edge sets should certainly be one respected by graphs and their matroids in these classes. 

At one extreme, a tree represents a graphic matroid in which each element is a component. 
At the other extreme we have the lifted-graphic matroid of Example \ref{ex:bad_all_loops_example}: a matroid consisting of a single parallel class of $n$ elements is represented by a graph with $n$ components. 
However, $U_{1,n}$ is also represented as a lifted-graphic matroid by the graph consisting of a single vertex with $n$ incident unbalanced loops. 
This turns out to be key. 

In none of the classes of graphic, lifted-graphic, nor frame matroids do the components of a matroid on ground set $E(G)$ necessarily correspond to components of the graph $G$. 
Each class contains matroids with many components defined on the ground set of a connected graph. 
Conversely, as we have seen, in the class of lifted-graphic matroids a connected matroid may be represented by a graph with many components. 
For cycle matroids of graphs, Whitney's 2-isomorphism theorem characterises the situation; a connected matroid must be represented by a connected graph. 
In the class of lifted-graphic matroids, though there are connected matroids represented by a disconnected graph, in this case we may always find a connected graph representing the matroid (Theorem \ref{thm:liftedgraphic_has_connected_graph}(2) below). 
{Such a connected graph is obtained by the Whitney-operation of identifying pairs of vertices, one from each of two distinct components of the graph.} 
In the class of frame matroids, as is the case with cycle matroids of graphs, a connected matroid cannot be represented by a disconnected graph. 
We thus lose nothing from these classes by demanding that, 
{when defining a matroid $M$ on the set of edges of a graph $G$, there always be a graph $H$ Whitney equivalent to $G$ for which no component of $M$ contains edges from distinct components of $H$}. 

We therefore answer Zaslavsky's questions (1) and (2) subject to the condition that every component of a matroid on $E(G)$ be contained in a 
{connected} component of 
{some graph Whitney equivalent to} 
$G$. 
This condition turns out to yield a rich and interesting class.  
Indeed, we will see that the resulting class almost coincides with the class of quasi-graphic matroids; quasi-graphic matroids are 
the only 
{3-connected} 
matroids intermediate for biased graphs subject to this condition. 

Now let us return to the notion of a refinement of the partition $(\Bb,\Uu)$ of the set of cycles defining a biased graph. 
Let $G$ be a graph and let $(\Bb,\Ll,\Ff)$ be a partition of the cycles of $G$ into three sets.
Call such a tripartition $(\Bb,\Ll,\Ff)$ \emph{proper} if $\Bb$ obeys the theta property and every cycle in $\Ll$ meets every cycle in $\Ff$.
Let $\Tt$ be the set of all pairs $(G,(\Bb,\Ll,\Ff))$ consisting of a graph together with a proper tripartition of its cycles.
We show that there is a map $\mathbf{T}$ from $\Tt$ to the set of matroids such that $\mathbf{T}$ respects restrictions, and
for all pairs $(G,(\Bb,\Ll,\Ff))$ in $\Tt$ 
\[ \Ii\br{L\GB} \subseteq \Ii\br{\mathbf{T}(G,(\Bb,\Ll,\Ff))} \subseteq \Ii\br{F\GB}.\]

To define our map $\mathbf{T}$ requires the development of a few tools. 
Foremost, we need a better understanding of 
frameworks for quasi-graphic matroids. 
We will return to the topic of Zaslavsky's intermediate matroids and provide our answers to 
{Zaslavsky's} 
questions (1) and (2) in Section \ref{sec:Biased_graphic_matroids}. 

\section{Bracelet functions and tripartitions of cycles} 

In this section we give two explicit descriptions of quasi-graphic matroids in terms of the subgraphs induced by their circuits in a framework graph and prove Theorem \ref{fixed_2_thm:equivalences}. 
Efforts to understand which bracelets may be dependent and which may be independent lead to the notions of bracelet functions and proper tripartitions of cycles. 
We now explain these notions. 

\subsection{Bracelet functions} 

Let $(G,\Bb)$ be a biased graph.
We construct an auxiliary graph to capture the relationships between bracelets in $G$.
The \emph{cyclomatic number} $\beta(X)$ of a subset $X \subseteq E(G)$ is the minimum number of edges that must be removed from the induced subgraph $G[X]$ in order to obtain an acyclic subgraph.
Let $B$ and $B'$ be distinct bracelets of $G$.
Then $\beta(B \cup B') \geq 3$.
Moreover, since $B \cup B'$ is a union of cycles, $B \cup B'$ has no bridge.

\begin{obs} 
Every graph with no bridge and cyclomatic number three is a subdivision of one of the graphs in Figure \ref{F:CyclomaticNumber3} or of a graph obtained by contracting some edges of one of these graphs.
\end{obs}

\begin{proof}
This observation follows easily from the following results of Whitney  \cite{MR1501641}. 
Let $G$ be a bridgeless graph with at least two edges. Then: 
\begin{itemize} 
\item $G$ uniquely decomposes into its blocks $G_1, \ldots, G_k$ (these are the maximal 2-connected subgraphs of $G$, along with loops); 
\item the cyclomatic number of $G$ is the sum of the cyclomatic numbers of $G_1$, \ldots, $G_k$; 
\item $G$ is 2-connected if and only if $G$ has a proper ear decomposition (that is, an ear decomposition starting with a cycle in which each ear aside from the initial cycle has distinct endpoints); 
\item $G$ has cyclomatic number 1 if and only if $G$ is a cycle; 
\item the cyclomatic number of $G$ is equal to the number of ears in a ear decomposition of $G$. 
\end{itemize} 
Consider the blocks $G_1$, \ldots, $G_k$ of $G$. 
By Whitney's results above, each of these subgraphs of cyclomatic number 1 is a cycle, and each of cyclomatic number 2 is a theta. 
Thus it is straightforward to check that if $G$ has cyclomatic number 3 and $k \geq 2$, then $G$ is either a subdivision of one of the two graphs at right in Figure \ref{F:CyclomaticNumber3}, or a subdivision of a graph obtained from one of the graphs in Figure \ref{F:CyclomaticNumber3} by contracting some of its edges. 
If $G$ has cyclomatic number 3 and $k=1$, then a proper ear decomposition of $G$ has exactly three ears. 
Either the third ear in the decomposition meets just one of the previous ears in the decomposition, or the third ear meets both previous ears in the decomposition. 
In the first case $G$ is a subdivision of the graph second from left in Figure \ref{F:CyclomaticNumber3} or a subdivision of a graph obtained by contracting some of its edges. 
In the second case $G$ is a subdivision of $K_4$ or a subdivision of a graph obtained by contracting some of its edges. 
\end{proof}

Let $B$ and $B'$ be distinct bracelets of $G$.
Neither of the two graphs at left in Figure \ref{F:CyclomaticNumber3}, nor any contraction of either of these graphs, contains two distinct bracelet pairs. 
Hence if $\beta(B \cup B') = 3$ then $B \cup B'$ is a subdivision of one of the two graphs at right in Figure \ref{F:CyclomaticNumber3}{\color{black}, or a contraction of one of these graphs}.
\begin{figure}
\begin{center}
\includegraphics[page=1,scale=.8]{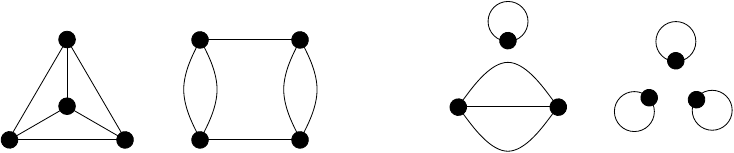}
\end{center}
\caption{Every bridgeless graph of cyclomatic number three is a subdivision of a contraction
of one of the graphs in this figure.}
\label{F:CyclomaticNumber3}
\end{figure}
The \emph{bracelet graph} $\mathscr B\GB$ of $(G,\Bb)$ is the graph with vertex set the collection of bracelets of $(G,\Bb)$ in which $BB'$ is an edge if and only if the cyclomatic number of $B \cup B'$ is 3.

A \emph{bracelet function} is a function from the set of bracelets of a biased graph $(G,\mc B)$ to the set $\{\mathsf{independent},\mathsf{dependent}\}$.
If $\chi$ is a bracelet function with the property that $\chi(B_1)=\chi(B_2)$ whenever $B_1$ and $B_2$ are in the same component of $\mathscr B\GB$, then $\chi$ is a \emph{proper} bracelet function.
Given a biased graph $\GB$ with bracelet function $\chi$, let $\mc C(G,\mc B,\chi)$ be the collection of edge sets of: balanced cycles, thetas with no cycle in $\Bb$, tight handcuffs, bracelets $B$ with
$\chi(B)=\mathsf{dependent}$, and loose handcuffs containing bracelets $B$ with $\chi(B)=\mathsf{independent}$. 
{\color{black}
The following two results say that whenever a collection $\Cc(G,\Bb,\chi)$ is the collection of circuits of a matroid, where either $G$ or the matroid is connected, then $\chi$ must be proper. 
The final result of this subsection says that, conversely, proper bracelet functions define matroids. 
}

\begin{theorem} \label{thm:BraceletFunctionPropriety}
Let $\GB$ be a biased graph with $G$ connected, 
and let $\chi$ be a bracelet function for $\GB$. 
If $\Cc(G,\Bb,\chi)$ is the set of circuits of a matroid, then $\chi$ is proper.
\end{theorem}

\begin{proof}
Suppose for a contradiction that $\Cc(G,\Bb,\chi)$ is the set of circuits of a matroid and $\chi$ is not proper. 
Then there are bracelets $B_1$ and $B_2$ adjacent in the bracelet graph of $\GB$ with $\chi(B_1)=\mathsf{independent}$ and $\chi(B_2)=\mathsf{dependent}$. 
Since $B_1$ and $B_2$ are adjacent in $\mathscr B\GB$ they share a cycle $C$ that is a component of $G[B_1 \cup B_2]$. 
Let $B_1=C_1 \cup C$ and $B_2=C_2 \cup C$. 
The subgraph $G[C_1 \cup C_2]$ is either a theta subgraph, tight handcuffs, or a bracelet. 
In the case that $G[C_1 \cup C_2]$ is a theta, tight handcuffs, or dependent bracelet, $G[C_1 \cup C_2]$ contains a circuit $Y \in \Cc(G,\Bb,\chi)$. 
If $G[C_1 \cup C_2]$ is an independent bracelet, and there is a path $P$ linking $C_1$ and $C_2$ while avoiding $C$, then $C_1 \cup P \cup C_2 \in \Cc(G,\Bb,\chi)$; put $Y=C_1 \cup P \cup C_2$. 
Let $e \in (Y \cap B_2)-B_1$. 
Since $B_2 \in \Cc(G,\Bb,\chi)$, by the circuit elimination axiom there is a circuit in $\Cc(G,\Bb,\chi)$ contained in $(Y \cup B_2)-e$. 
Since $(Y \cup B_2)-e$ has cyclomatic number two and contains $B_1$ this set is independent, a contradiction. 
So assume $G[C_1 \cup C_2]$ is an independent bracelet, and every path linking $C_1$ and $C_2$ meets $C$. 
Let $P$ be a minimal path linking $C$ and $C_1$. 
Then $B_1 \cup P \in \Cc(G,\Bb,\chi)$ is a circuit. 
Let $e \in C$. 
By the circuit elimination axiom there is a circuit contained in $(B_1 \cup P \cup B_2)-e$. 
But $(B_1 \cup P \cup B_2)-e$ has cyclomatic number two and contains $C_1 \cup C_2$ but no path from $C_1$ to $C_2$, so again this set is independent, a contradiction. 
\end{proof}

{\color{black}
\begin{theorem} \label{new_Thm_2point1}
Let $(G, \Bb)$ be a biased graph, and let $\chi$ be a bracelet function for $(G,\Bb)$. 
If $\Cc(G, \Bb, \chi)$ is the set of circuits of a connected matroid, then $\chi$ is proper. 
\end{theorem}

Theorem \ref{new_Thm_2point1} follows from the following lemma. 

\begin{lemma} \label{bracelets_in_diff_components_are_circuits}
Let $(G, \Bb)$ be a biased graph, let $\chi$ be a bracelet function for $(G,\Bb)$. 
Suppose $\Cc(G, \Bb, \chi)$ is the set of circuits of a connected matroid $M$, but that $G$ 
has two components $G_1$ and $G_2$. 
Let $C_1$ in $G_1$ and $C_2$ in $G_2$ be a pair of unbalanced cycles.  
Then $\chi(C_1 \cup C_2) = \mathsf{dependent}$, so $C_1 \cup C_2$ is a circuit of $M$.
\end{lemma}

\begin{proof}
Suppose to the contrary that $\chi(C_1 \cup C_2) = \mathsf{independent}$. 
Let $e_i$ be an element in cycle $C_i$, $i \in \{1,2\}$. 
As $M$ is connected, there is a circuit of $M$ containing both $e_1$ and $e_2$. 
As there is no path in $G$ linking $e_1$ and $e_2$, this circuit is a bracelet 
$C_1' \cup C_2'$, where, for each $i \in \{1,2\}$, $C_i'$ is an unbalanced cycle in $G_i$ containing $e_i$. 

\begin{claim}
Let $C$, $C'$, $C''$ be unbalanced cycles with $C$ disjoint from $C' \cup C''$ and where $C'$ and $C''$ share at least one 
edge. 
If $\chi(C \cup C') = \mathsf{dependent}$ then also $\chi(C \cup C'') = \mathsf{dependent}$. 
\end{claim}

\begin{proof}[Proof of Claim]
Let $e$ be an edge of $C$, and let $D$ be a circuit of $M/C''$ contained in $C \cup C'$ containing $e$. 
Let $D'$ be a circuit of $M$ with $D \subseteq D' \subseteq D \cup C''$. 
We claim that $D$ is in fact contained in $C$. 
For suppose to the contrary that $D$ contains an edge $f$ of $C'$. 
Let $P$ be the path in $C'$ 
containing $f$ 
whose endpoints are in $C''$ but none of whose internal vertices are in $C''$. 
As $G[D']$ has no vertex of degree one, 
all edges of $P$ are contained in $D'$ and so in $D$. 
Since $C'' \cup P$ is dependent in $M$ while $C''$ is not, 
this implies that $D$ properly includes a circuit of $M/C''$, a contradiction. 
Thus $D$ contains no edge of $C'$, 
so $D'$ is contained in $C \cup C''$. 
\end{proof}

As $\chi(C_1' \cup C_2') = \mathsf{dependent}$, by the claim, 
also $\chi(C_1' \cup C_2) = \mathsf{dependent}$. 
Applying the claim again, we conclude $\chi(C_1 \cup C_2) = \mathsf{dependent}$. 
\end{proof}

\begin{proof}[Proof of Theorem \ref{new_Thm_2point1}]
Suppose for a contradiction that $\Cc(G,\Bb,\chi)$ is the set of circuits of a matroid and $\chi$ is not proper. 
Then there are bracelets $B_1$ and $B_2$ adjacent in the bracelet graph of $(G,\Bb)$ with $\chi(B_1)=\mathsf{independent}$ and $\chi(B_2)=\mathsf{dependent}$. 
Let $B_1=C_1 \cup C$ and $B_2=C_2 \cup C$. 
The subgraph $G[C_1 \cup C_2]$ is either a theta subgraph, tight handcuffs, or a bracelet. 
In the case that $G[C_1 \cup C_2]$ is a theta, tight handcuffs, or dependent bracelet, $G[C_1 \cup C_2]$ contains a circuit $Y \in \Cc(G,\Bb,\chi)$. 
If $G[C_1 \cup C_2]$ is an independent bracelet, and there is a path $P$ linking $C_1$ and $C_2$ while avoiding $C$, then $C_1 \cup P \cup C_2 \in \Cc(G,\Bb,\chi)$; put $Y=C_1 \cup P \cup C_2$. 
Let $e \in (Y \cap B_2)-B_1$. 
Since $B_2 \in \Cc(G,\Bb,\chi)$, by the circuit elimination axiom there is a circuit in $\Cc(G,\Bb,\chi)$ contained in $(Y \cup B_2)-e$. 
Since $(Y \cup B_2)-e$ has cyclomatic number two and contains $B_1$ this set is independent, a contradiction. 

So assume $G[C_1 \cup C_2]$ is an independent bracelet, 
and if there is a path linking $C_1$ and $C_2$, then this path meets $C$. 
Suppose first that such a path exists. 
Let $P$ be a minimal path linking $C$ and $C_1$. 
Then $B_1 \cup P \in \Cc(G,\Bb,\chi)$ is a circuit. 
Let $e \in C$. 
By the circuit elimination axiom there is a circuit contained in $(B_1 \cup P \cup B_2)-e$. 
But $(B_1 \cup P \cup B_2)-e$ has cyclomatic number two and contains $C_1 \cup C_2$ but no path from $C_1$ to $C_2$, so again this set is independent, a contradiction. 
The final possibility is that $C_1 \cup C_2$ forms an independent bracelet and there is no path in $G$ linking $C_1$ and $C_2$. 
But Lemma \ref{bracelets_in_diff_components_are_circuits}  implies that $C_1 \cup C_2$ is a circuit, so this is impossible. 
\end{proof}
}

\begin{theorem} \label{thm:bracelet_defn_of_matroid}
Let $\GB$ be a biased graph, and let $\chi$ be a proper bracelet function for $\GB$. 
Then $\Cc(G,\Bb,\chi)$ is the set of circuits of a matroid.
\end{theorem}

\begin{proof}
It is clear that no element of $\Cc(G,\Bb,\chi)$ is properly contained in another, so we just need to show that the collection satisfies the circuit elimination axiom.
Let $C_1$ and $C_2$ be two elements of $\mc C(G,\mc B,\chi)$ and suppose $e \in C_1 \cap C_2$.
Since $\beta(C_1\cup C_2)$ is strictly greater than each of $\beta(C_1)$ and $\beta(C_2)$, $\beta(C_1\cup C_2) \geq 2$.
We consider the four cases $\beta(C_1 \cup C_2) = 2, 3, 4$, and $\beta(C_1 \cup C_2) \geq 5$.

The case $\beta(C_1 \cup C_2) =2$ is straightforward: 
Since $\beta(C_1\cup C_2)$ is strictly greater than each of $\beta(C_1)$ and $\beta(C_2)$, $\beta(C_1\cup C_2)=2$ if and only if $C_1$ and $C_2$ are both balanced cycles whose union is a theta graph.
Because $\Bb$ satisfies the theta property, the cycle $C_3$ contained in $\br{C_1 \cup C_2}-e$ is balanced, so $C_3 \in \Cc(G,\Bb,\chi)$.

The following two simple observations will be used for the remaining cases.
\begin{itemize}
\item A connected biased subgraph $H$ with $\beta(H) \geq 2$ contains an element of $\mc C(G,\mc B,\chi)$.
\item $C_1 \cup C_2$ has at most three components (else $C_1 \cap C_2 = \emptyset$).
\end{itemize}

Together these observations imply that if $\beta(C_1\cup C_2) \geq 5$ then $\br{C_1 \cup C_2}-e$ has a component $H$ with $\beta(H) \geq 2$, and so contains an element of $\Cc(G,\Bb,\chi)$.

So suppose now $\beta(C_1\cup C_2)=4$.
We may assume that $C_1 \cup C_2$ has three components, since otherwise $\br{C_1 \cup C_2}-e$ has a component with cyclomatic number at least two, which therefore contains an element of $\Cc(G,\Bb,\chi)$.
Let $A_1$, $A_2$, $A_3$ be the components of $C_1 \cup C_2$; without loss of generality assume $\beta(A_1) = \beta(A_2) = 1$ and $\beta(A_3) = 2$.
Since $C_1$ and $C_2$ share the element $e$, this implies both of $C_1$ and $C_2$ are dependent bracelets and that $e \in A_3$. 
Hence $A_1 \cup A_2$ is a bracelet, and $\br{A_1 \cup A_2 \cup C_1}$ consists of three pairwise vertex disjoint unbalanced cycles. 
Thus $A_1 \cup A_2$ and $C_1$ are adjacent in the bracelet graph of $\GB$. 
Since $\chi(C_1)=\mathsf{dependent}$, this implies $\chi(A_1 \cup A_2)=\mathsf{dependent}$. 
That is, $A_1 \cup A_2 \subseteq \br{C_1 \cup C_2}-e$ is an element of $\Cc(G,\Bb,\chi)$. 

Finally suppose $\beta(C_1 \cup C_2) = 3$.
We may assume that $C_1 \cup C_2$ has at least two  components, else $\br{C_1 \cup C_2}-e$ has a component with cyclomatic number at least two, which therefore contains an element of $\Cc(G,\Bb,\chi)$.
Suppose first $C_1\cup C_2$ has three components.
Then each component is a single cycle, and both $C_1$ and $C_2$ are dependent bracelets.
Let $A_1, A_2, A_3$ be the components of $C_1 \cup C_2$, and suppose without loss of generality that $C_1 = A_1 \cup A_3$ and $C_2 = A_2 \cup A_3$.
Then as in the case above, $e \in E(A_3)$ and $A_1 \cup A_2$ is a bracelet.
Since $C_1$, $C_2$, and $A_1 \cup A_2$ are mutually adjacent in the bracelet graph, $\chi(A_1 \cup A_2) =  \mathsf{dependent}$.
Thus $A_1 \cup A_2 \subseteq \br{C_1 \cup C_2}-e$ is a bracelet in $\Cc(G,\Bb,\chi)$.
Now suppose $C_1 \cup C_2$ has just two components, $A_1$ and $A_2$.
Then, without loss of generality, $A_1$ is an unbalanced cycle, $C_1$ is a bracelet, and $A_2$ is either a theta or pair of handcuffs.
If $e \in A_1$ then since $\beta(A_2)=2$ there is an element of $\Cc(G,\Bb,\chi)$ in $\br{C_1 \cup C_2} - e$.
So suppose $e \in A_2$.
Since all of the bracelets contained in $C_1 \cup C_2$ are mutually adjacent in the bracelet graph, all are assigned $\mathsf{dependent}$ by $\chi$.
One of these dependent bracelets is contained in $\br{C_1 \cup C_2}-e$.
\end{proof}

When $\Cc(G,\Bb,\chi)$ is the set of circuits of a matroid, we denote this matroid by $M(G,\Bb,\chi)$. 
If $M=M(G,\Bb,\chi)$ for some biased graph $\GB$ with bracelet function $\chi$, then we say $G$ \emph{is a graph for} $M$. 

\subsection{Proper tripartitions} 

We now consider refinements of the partition $(\Bb,\Uu)$ of the cycles of a biased graph.  
Let $(\Bb,\Ll,\Ff)$ be a partition of the cycles of a graph $G$.
We say that $(\Bb,\Ll,\Ff)$ is a \emph{proper tripartition} if the cycles in $\Bb$ obey the theta property and every cycle in $\Ll$ meets every cycle in $\Ff$.
Given a graph $G$ and a proper tripartition $(\Bb,\Ll,\Ff)$ of its cycles, let $\Cc(G,\Bb,\Ll,\Ff)$ be the collection of subsets of $E(G)$ consisting of cycles in $\Bb$, thetas with no cycle in $\Bb$, tight handcuffs with neither cycle in $\Bb$, bracelets with both cycles in $\Ll$, and loose handcuffs with both cycles in $\Ff$.

\begin{theorem} \label{thm:tripartitions_give_matroids}
Let $G$ be a graph and let $(\Bb,\Ll,\Ff)$ be a proper tripartition of its cycles.
Then $\Cc(G,\Bb,\Ll,\Ff)$ is the set of circuits of a matroid.
\end{theorem}

\begin{proof}
Define a bracelet function $\chi$ for $\GB$ as follows.
For each bracelet $B = C \cup C'$, define $\chi(B) = \mathsf{dependent}$ if $C$ and $C'$ are both in $\Ll$, and $\chi(B) = \mathsf{independent}$ if $C$ and $C'$ are both in $\Ff$.
Since $(\Bb,\Ll,\Ff)$ is a proper tripartition, no bracelet of $G$ contains a cycle in $\Ll$ while its other cycle is in $\Ff$.
This implies:
\begin{itemize}
\item $\chi$ is defined on every bracelet of $G$, and
\item  in the bracelet graph of $\GB$ no bracelet whose cycles are both in $\Ll$ is adjacent to any bracelet whose cycles are both in $\Ff$.
\end{itemize}
Thus $\chi$ is constant on each component of $\mathscr B\GB$; that is, $\chi$ is a proper bracelet function.
Hence $\Cc(G,\Bb,\Ll,\Ff) = \Cc(G,\Bb,\chi)$, which by Theorem \ref{thm:bracelet_defn_of_matroid} is the set of circuits of a matroid.
\end{proof}

Given a graph $G$ and a proper tripartition $(\Bb,\Ll,\Ff)$ of its cycles, denote the matroid of Theorem \ref{thm:tripartitions_give_matroids} by $M(G,\Bb,\Ll,\Ff)$. 
As for matroids arising from biased graphs with bracelet functions, if $M=M(G,\Bb,\Ll,\Ff)$ for some graph $G$ with proper tripartition $(\Bb,\Ll,\Ff)$, then we say $G$ \emph{is a graph for} $M$. 

We need just a couple more results before we can prove Theorem \ref{fixed_thm:equivalences}. 
First, we need the rank function for $M(G,\Bb,\Ll,\Ff)$. 
Let $G$ be a graph, and let $(\Bb, \mc L, \mc F)$ be a proper tripartition of the cycles of $G$. 
For a subset $X \subseteq E(G)$, denote by $V(X)$ the set of vertices incident to an edge in $X$ and by $c(X)$ the number of components of the induced biased subgraph $G[X]$. 
A subgraph that contains no unbalanced cycle is \emph{balanced}; otherwise it is \emph{unbalanced}. 
Denote by $b(X)$ the number of balanced components of $G[X]$.
Let
\[
l(X) = \begin{cases} 1 &\text{ if $G[X]$ contains a cycle in $\Ll$} \\ 0 &\text{ otherwise.}\end{cases}
\]

\begin{lemma} \label{lem:rank}
Let $G$ be a graph and let $(\Bb, \mc L, \mc F)$ be a proper tripartition of its cycles.
The rank of a subset $X \subseteq E(G)$ in $M(G,\Bb,\Ll,\Ff)$ is
\[
r(X) =
	\begin{cases}
		|V(X)| - b(X) &\text{ if } G[X] \text{ contains a cycle in } \mc F \\
		|V(X)| - c(X) + l(X) &\text{ otherwise.}
	\end{cases}
\]
\end{lemma}

\begin{proof}
Let $X \subseteq E(G)$, and suppose $X$ contains a cycle in $\Ff$.
Since every cycle in $\Ll$ meets every cycle in $\Ff$, $X$ cannot include two cycles $\Ll$ in different components. 
A maximal independent set in $G[X]$ consists of a spanning tree of each balanced component together with a spanning subgraph of each unbalanced component of cyclomatic number one in which the unique cycle is unbalanced. 
Thus $r(X) = |V(X)| - b(X)$.
Now suppose $X$ does not contain a cycle in $\Ff$.
Then a maximal independent set in $G[X]$ has cyclomatic number at most one, so 
\[ r(X) = |V(X)| - c(X) + l(X).
\qedhere \]
\end{proof}

We use the following straightforward corollary of a result of Tutte \cite[(4.34)]{Tutte:Lectures}, to prove the following lemma. 

\begin{theorem} \label{T:TuttesPathTheorem}
Let $\GB$ be a biased graph and suppose $G$ is 2-connected.
If $C$ and $C'$ are unbalanced cycles in $(G,\mc B)$, then there are unbalanced cycles $C_1,\ldots,C_n$ such that $C_1=C$, $C_n=C'$, and each $C_i\cup C_{i+1}$ is a theta.
\end{theorem}

\begin{lemma} \label{lem:everycyclecontainingChassamechivalue}
Let $\GB$ be a biased graph and let $\chi$ be a proper bracelet function for $\GB$.
Then for every unbalanced cycle $C$, every bracelet containing $C$ is assigned the same value by $\chi$.
\end{lemma}

\begin{proof}
Consider an unbalanced cycle $C$ and two bracelets $B=C\cup C'$ and $B'=C\cup C''$.
If $C'$ and $C''$ are in different blocks of $G-V(C)$, then $B$ and $B'$ are adjacent in the bracelet graph $\mathscr B\GB$.
If $C'$ and $C''$ are in the same block of $G-V(C)$, then by Theorem \ref{T:TuttesPathTheorem} there is a path of bracelets $B_1, \ldots, B_n$ in $\mathscr B\GB$ such that $C$ is in each bracelet, $B_1=B$, and $B_n=B'$.
In either case $B$ and $B'$ are in the same component of $\mathscr B\GB$. 
Since $\chi$ is proper, $\chi(B)=\chi(B')$. 
\end{proof}

We also require the following lemma from \cite{JGT:JGT22177}. 

\begin{lemma}[\cite{JGT:JGT22177}, Lemma 2.6] \label{lem:GGW2.6}
Let $G$ be a framework for a matroid $M$. 
If $H$ is a subgraph of $G$ with $|E(H)| > |V(H)|$, then $E(H)$ is dependent in $M$. 
\end{lemma}

We can now prove Theorem \ref{fixed_thm:equivalences}. 
When it is important that the distinction be clear, an edge whose endpoints are distinct is called a \emph{link}. 

\begin{proof}[Proof of Theorem \ref{fixed_thm:equivalences}]
($1. \Rightarrow 2.$) 
Suppose the circuits of $M$ are given by $\Cc(G,\Bb,\chi)$ for some proper bracelet function $\chi$.
Let $\Ll$ be the collection of cycles $C$ for which there is a bracelet $B$ containing $C$ with $\chi(B) = \mathsf{dependent}$; let $\Ff$ be the collection of unbalanced cycles not in $\Ll$. 
Suppose for a contradiction that $(\Bb,\Ll,\Ff)$ is not proper.
Thus there is a pair of vertex disjoint unbalanced cycles $C, C'$ such that $C \in \Ll$ and $C' \in \Ff$. 
Then $C \cup C'$ is a bracelet.
Since $C \in \Ll$ there is an unbalanced cycle $D$ that is vertex disjoint from $C$ such that $\chi(C \cup D)=\mathsf{dependent}$. 
Hence by Lemma \ref{lem:everycyclecontainingChassamechivalue}, $\chi(C \cup C') = \mathsf{dependent}$. 
But then $C' \in \Ll$, a contradiction. 

($2. \Rightarrow 3.$) 
Suppose $M=M(G,\Bb,\Ll,\Ff)$ where $(\Bb,\Ll,\Ff)$ is a proper tripartition of the cycles of $G$. 
Then $G$ is a framework for $M$: 
The condition $E(G)=E(M)$ is immediate. 
That $r(E(H)) \leq |V(H)|$ for every component $H$ of $G$ follows from Lemma \ref{lem:rank}.
The condition that for every circuit $C$ of $M$, $G[C]$ has at most two components holds because no element of $\Cc(G,\Bb,\Ll,\Ff)$ has more than two components. 
Finally, let $v \in V(G)$ and let $e$ be a link incident to $v$.
Since $\Cc(G,\Bb,\Ll,\Ff)$ is the set of circuits of $M$,
there is no circuit containing $e$ contained in $E(G-v) \cup \{e\}$, so
\[
\cl(E(G-v)) \subseteq E(G-v) \cup \loops(v)
\]
as required.

($3. \Rightarrow 1.$)
Let $\chi$ be the bracelet function defined according to the independence or dependence of each bracelet in $M$; that is, for each bracelet $B$ of $G$, define $\chi(B) = \mathsf{independent}$ if and only if $B \in \Ii(M)$. 
We show that the collection $\Cc(G,\Bb,\chi)$ is the collection of circuits of $M$. 

We first show $\Cc(G,\Bb,\chi) \subseteq \Cc(M)$. 
Let $X \in \Cc(G,\Bb,\chi)$. 
If $X$ is a balanced cycle then $X$ is a circuit of $M$ by definition. 
Suppose $X$ is  the edge set of a theta or tight handcuff with no cycle in $\Bb$.
By Lemma \ref{lem:GGW2.6} every subgraph $H$ of $G$ with $|E(H)| > |V(H)|$ has $E(H)$ dependent in $M$, so $X$ is dependent in $M$. 
Let $Y \subseteq X$ be a circuit of $M$. 
Since $X$ does not contain a cycle in $\Bb$, $Y$ is not a balanced cycle. 
Hence by Lemma \ref{lem:GGWquasicircuits}, $Y=X$. 
Thus $X \in \Cc(M)$.
Now suppose $X$ is the edge set of a bracelet of $G$. 
A bracelet is in $\Cc(G,\Bb,\chi)$ if and only if it is dependent in $M$. 
Let $Y \subseteq X$ be a circuit of $M$. 
By Lemma \ref{lem:GGWquasicircuits}, $Y=X$, so again $X \in \Cc(M)$. 
Finally, suppose $X$ is a loose handcuff in $G$. 
Again, by Lemma \ref{lem:GGW2.6}, $X$ is dependent in $M$ and so contains a circuit $Y$. 
By Lemma \ref{lem:GGWquasicircuits}, either $Y=X$ or $Y$ is the bracelet $B$ contained in $X$. 
But $\chi(B) = \mathsf{independent}$ so by definition $B$ is independent in $M$. 
Thus $X=Y \in \Cc(M)$. 

We now show that $\Cc(M) \subseteq \Cc(G,\Bb,\chi)$.
Let $X \in \Cc(M)$. 
By Lemma \ref{lem:GGWquasicircuits}, in $G$, $X$ is either (i) a balanced cycle, (ii) a connected subgraph with no cycle in $\Bb$, minimum degree at least two, and exactly one more edge than vertices, or (iii) a bracelet. 
If $X \in \Bb$, then $X \in \Cc(G,\Bb,\chi)$. 
If $X$ is one of the subgraphs of the form (ii), then $X$ is either a theta with no cycle in $\Bb$, tight handcuffs, or loose handcuffs. 
In the first two cases, $X \in \Cc(G,\Bb,\chi)$. 
If $X$ is a pair of loose handcuffs, then the bracelet properly contained in $X$ is independent, so $\chi(X)=\mathsf{independent}$ and $X \in \Cc(G,\Bb,\chi)$. 
Finally, suppose $X$ is a bracelet. 
Then $\chi(X)=\mathsf{dependent}$ so again $X \in \Cc(G,\Bb,\chi)$. 

Since $\Cc(G,\Bb,\chi)$ is the set of circuits of $M$, $M = M(G,\Bb,\chi)$, and by Theorem \ref{new_Thm_2point1}, $\chi$ is proper.
\end{proof}

{\color{black}
Observe that the requirement in Theorem \ref{fixed_thm:equivalences} that $M$ be connected is necessary: 
the class of quasi-graphic matroids is closed under direct sums, but in general neither the bracelet function construction nor the triparition construction permit a direct sum. 
To see this, let $M$ be the direct sum of a quasi-graphic matroid $N$ that is not frame and a quasi-graphic matroid $N'$ that is not lifted-graphic. 
Let $H$ be a framework for $N$ and let $H'$ be a framework for $N'$, and let $G$ be the disjoint union of $H$ and $H'$. 
Then $G$ is clearly a framework for $M$, but clearly there is no proper bracelet function $\chi$ nor proper tripartition $(\Bb,\Ll,\Ff)$ of the cycles of $G$ for which $M = M(G,\Bb,\chi)$ or $M = M(G,\Bb,\Ll,\Ff)$. 
However, we may remove the condition that the matroid $M$ be connected provided we replace it with the condition that the graph $G$ be connected: 

\begin{theorem}%[corrected statement of Theorem 1.1 \cite{MR4037634}] 
\label{what_we_actually_proved}
Let $M$ be a matroid and let $(G,\Bb)$ be a biased graph with $E(G) = E(M)$. 
Consider the following statements. 
\begin{enumerate}%[label=\textup{(\arabic*)}]
\item There is a proper bracelet function $\chi$ for $G$ such that $M=M(G,\Bb,\chi)$. 
\item There is a proper tripartition $(\Bb,\Ll,\Ff)$ of the cycles of $G$ such that $M=M(G,\Bb,\Ll,\Ff)$. 
\item $M$ is quasi-graphic with framework $G$ and $\Bb$ is the set of cycles of $G$ that are circuits of $M$. 
\end{enumerate} 
Statements $(1)$ and $(2)$ are equivalent; 
% Each of $(1)$ and $(2)$ imply $(3)$. 
% Either of $(1)$ or $(2)$ implies $(3)$. 
$(1)$ implies $(3)$; $(2)$ implies $(3)$. 
If $G$ is connected, then $(3)$ implies $(1)$ and $(2)$. 
In particular, if $G$ is connected, then $(1)$, $(2)$, and $(3)$ are equivalent.
\end{theorem}

\begin{proof}%[Proof of Theorem \ref{what_we_actually_proved}]
The implications (1) $\Rightarrow$ (2) and (2) $\Rightarrow$ (3) in the proof of Theorem \ref{fixed_thm:equivalences} require neither $M$ nor $G$ be connected. 
The implication (2) $\Rightarrow$ (1) is proved in the course of the proof of Theorem \ref{thm:tripartitions_give_matroids}, also with no connectivity condition. 
The proof that if $G$ is connected then (3) $\Rightarrow$ (1) is identical to that of the case that (3) $\Rightarrow$ (1) in the proof of Theorem \ref{fixed_thm:equivalences}, except in the final sentence, where rather than relying on Theorem \ref{new_Thm_2point1} to deduce that $\chi$ is proper, we apply Theorem \ref{thm:BraceletFunctionPropriety} to deduce that $\chi$ is proper. 
Thus provided $G$ is connected, statements (1), (2), and (3) are equivalent regardless of whether or not $M$ is connected. 
\end{proof}

Observe that Theorem \ref{what_we_actually_proved} really does describe quasi-graphic matroids not captured in Theorem \ref{fixed_thm:equivalences}: 
Let $M$ be a matroid represented by a biased graph $(G,\Bb)$ where $G$ has a cut vertex $v$. 
Let $(X,Y)$ be a bipartition of $E(G)$ with $V(X) \cap V(Y) = \{v\}$, and where $G[Y]$ is balanced. 
Then $M = M(G,\Bb,\chi)$ is a disconnected quasi-graphic matroid. 
Of course, we may also describe this matroid by applying Theorem \ref{fixed_thm:equivalences} to each of its connected components and considering the disjoint union of the frameworks of each connected component. 
}

\subsection{Degenerate tripartitions} 

We close this section with a useful observation about proper tripartitions. 
Let $G$ be a graph and let $(\Bb, \mc L, \mc F)$ be a proper tripartition of the cycles of $G$.
Call the collection $\mc L$ (respectively $\mc F$) \emph{degenerate} if $\Ll$ (resp.\ $\Ff$) is empty or no two cycles in $\Ll$ (resp.\ $\Ff$) are vertex disjoint.
Write $\Uu = \Ll \cup \Ff$. 
Observe that: 
\begin{itemize} 
\item If $\mc L$ is degenerate then $M(G, \Bb, \mc L, \mc F) = M(G,\Bb,\emptyset, \Uu) = F(G,\Bb)$. 
\item If $\mc F$ is degenerate then $M(G, \Bb, \mc L, \mc F) = M(G,\Bb,\Uu,\emptyset) = L(G,\Bb)$.
\end{itemize} 
Call the tripartition $(\Bb,\Ll,\Ff)$ \emph{degenerate} if one of $\Ll$ or $\Ff$ is degenerate. 

\section{Examples} 
\label{sec:examples} 

Examples of quasi-graphic matroids that are neither frame nor lifted-graphic do not, perhaps, easily spring to mind. 
Example \ref{ex:bicircularK_n} below is a non-example which may provide some intuition for why this may be. 

\begin{example} \label{ex:bicircularK_n}
The bracelet graph of $(K_n, \emptyset)$ is connected, so every bracelet function $\chi$ for $(K_n, \emptyset)$ is constant.
Hence if $M = M(K_n, \emptyset, \chi)$ then $M$ is either the lift matroid $L(K_n,\emptyset)$ or frame matroid $F(K_n,\emptyset)$.
\end{example}

The proper tripartition construction of Theorem \ref{fixed_thm:equivalences} makes finding examples a little easier. 

\begin{example} \label{ex:signed_graph} 
Let $G$ be a graph, and let $C$ be a 4-cycle in $G$. 
Let the edges of $C$ in cyclic order be $e_1, e_2, e_3, e_4$. 
Let 
\begin{itemize} 
\item $\Bb$ be the collection of cycles of $G$ that meet $C$ in an even number of edges, 
\item $\Ll$ be the set of cycles 
that meet $C$ in just $e_1$ or just $e_3$, 
\item $\Ff$ be the cycles 
that meet $C$ in just $e_2$, just $e_4$, or in precisely three edges. 
\end{itemize} 
Since $\Bb$ obeys the theta property and  every cycle in $\Ll$ meets every cycle in $\Ff$, $(\Bb,\Ll,\Ff)$ is a proper tripartition of the cycles of $G$. 
Thus $M(G,\Bb,\Ll,\Ff)$ is quasi-graphic. 
As long as $G$ has sufficiently many vertices and is sufficiently connected neither $\Ll$ nor $\Ff$ is degenerate. 
In particular, one can choose $G$ to be a complete graph on at least 8 vertices to obtain a quasi-graphic matroid 
with a representation with neither $\Ll$ nor $\Ff$ degenerate. 
Setting $|V(G)|=n$, this yields a family of quasi-graphic matroids of rank $n$ with $O(n^2)$ elements. 
Observe however that for each matroid $M(G,\Bb,\Ll,\Ff)$ in this family, deleting any one of the four vertices in $V(C)$ yields a subgraph with a degenerate tripartition. 
\end{example}

\begin{example}
Let $G$ be the graph obtained as follows from the complete bipartite graph $K_{a,b}$, with bipartition $(A, B)$ where
$|A|=a$, $|B|=b$. Add the edges of an $a$-cycle $C$ on $A$ and add the edges of a $b$-cycle $C'$ on $B$ to obtain graph
$G$. The bracelet $C \cup C'$ is an isolated vertex in the bracelet graph for $(G,\emptyset)$. Thus if we let one of $\mc L$ and $\mc F$ be $\{C,C'\}$ and the other consist of the remaining cycles in $G$, then the tripartition $(\emptyset,\mc L,\mc F)$ is a proper tripartition that is non-degenerate. 
All of these graphs $G$ with tripartitions have the property that for every vertex $v$ in $V(G)$, removing $v$ yields a subgraph with a degenerate tripartition.
\end{example}

\begin{example}
Let $G$ be a graph that is embedded in the torus and let $\mc B$ be the collection of contractible cycles of $G$.
Each non-contractible cycle is in a homology class $(a,b) \in \bb Z \times \bb Z$ where $(a,b)$ is a relatively prime pair.
Cycles in two different homology classes must intersect.
Thus each bracelet consists of a pair of cycles in the same homology class, and two bracelets from distinct homology classes cannot be adjacent in the bracelet graph of $\GB$.
Thus the number of connected components of the bracelet graph is the number of homology classes containing bracelets.
For the $2m\times 2m$ torus grid or any embedded graph containing it as a minor, there is a bracelet in each of the homology classes $(1,k)$, $(-1,k)$, $(k,1)$, and $(-k,1)$ for $k \in \{0,\ldots,m-1\}$
(see Figure \ref{F:TorusGrid}) so this number is at least $4m-4$.
\begin{figure}[tbp]
\begin{center}
\includegraphics[page=3,scale=0.45]{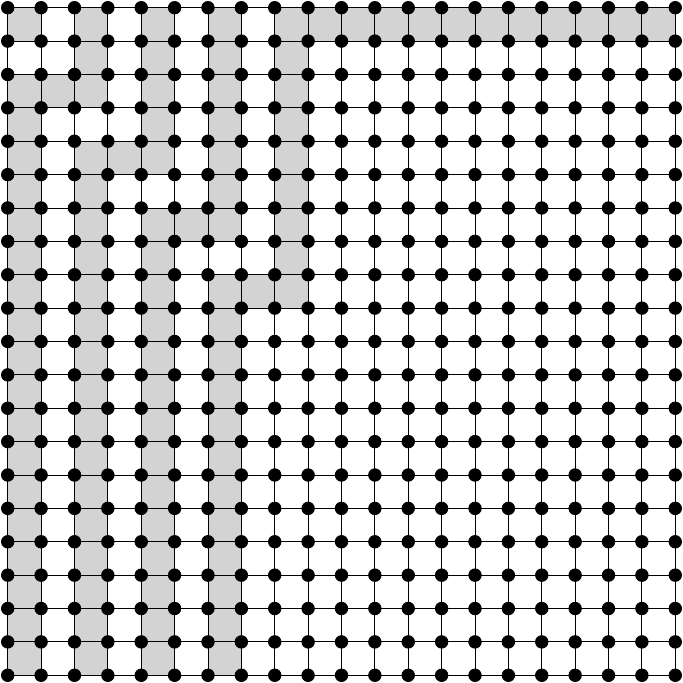}\qquad
\includegraphics[page=2,scale=0.45]{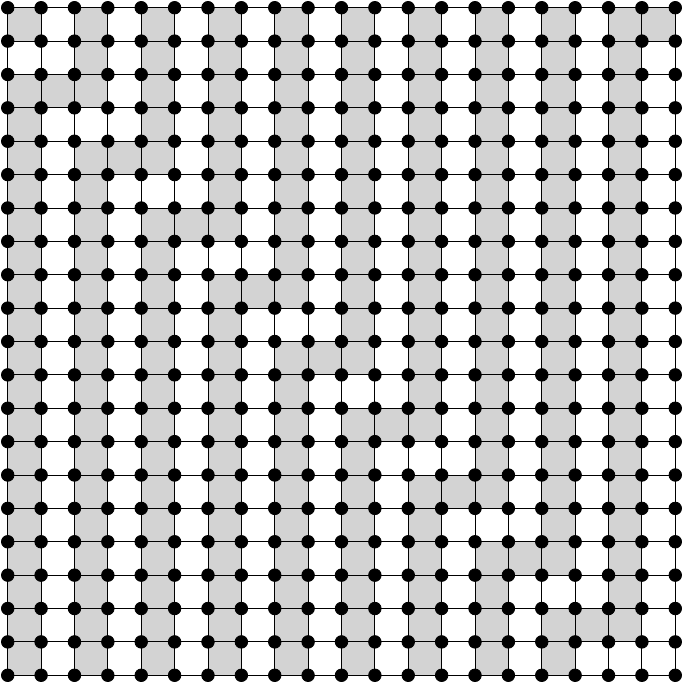}
\end{center}
\caption{In a $2m\times 2m$ torus grid there is a bracelet in each of homology classes $(1,0),(1,1),\ldots,(1,m-1)$. 
In each torus grid above, the pair of cycles bounding the union of the shaded faces form a bracelet.} 
\label{F:TorusGrid}
\end{figure}
Thus graphs embedded on the torus yield a family of biased graphs with $n$ vertices and $O(n)$ edges whose bracelet graphs can have $O(\sqrt n)$ components. So an $n$-vertex biased graph $(G,\mc B)$ coming from an embedding in the torus 
may potentially yield 
in the order of $2^{\sqrt{n}}$ pairwise non-isomorphic quasi-graphic matroids. 

In contrast to the families described in the previous two examples, there is no finite bound $b$ such that from every graph in the family it is possible to remove at most $b$ vertices to obtain a subgraph with a degenerate tripartition. 
\end{example}

The following example shows that there are frame matroids and lifted-graphic matroids represented by graphs with non-degenerate tripartitions. 

\begin{example} \label{ex:differingreps}
Each matroid in this example is sparse paving of rank four. 
Thus to check the claimed equalities, it is enough to check that the matroids have the same circuits of size four.

(a) Let $G_1$ be the graph at left in Figure \ref{fig:Quasigraphicframeexample}, with 
%$\Bb = \{xyac, xzac, xybd, xzbd, wxyz\}$. 
$\Bb_1 = \{wxyz, \allowbreak acwy, \allowbreak acxz, \allowbreak bdwy, \allowbreak bdxz\}$. 
Let $G_2$ be the graph at right in Figure \ref{fig:Quasigraphicframeexample}, with  
%$\Bb' = \{wxab, yzab, xycd, wzcd\}$, 
$\Bb_2 = \{abwx, \allowbreak abyz, \allowbreak cdxy, \allowbreak cdwz\}$, 
$\Ff_2 = \{ac, bd\}$, and $\Ll_2$ the set of unbalanced cycles not in $\Ff_2$. 
\begin{figure}[tbp] 
\begin{center} 
\includegraphics[scale=1]{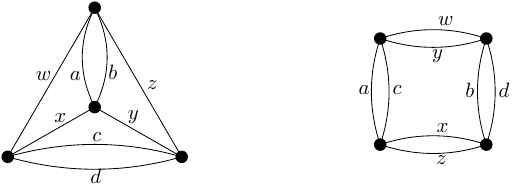}
\end{center} 
\caption{Example \ref{ex:differingreps}(a). Two graphs for a frame matroid.}
\label{fig:Quasigraphicframeexample} 
\end{figure} 
Then $F(G_1, \Bb_1) = M(G_2,\Bb_2,\Ll_2,\Ff_2)$. 

(b) Let $G_1$ be the graph obtained by adding an edge in parallel with each edge of $C_4$, labelled as shown at left in Figure \ref{fig:Quasigraphicframeliftexample1}, with $\Bb_1 = \{bcxy, adwz, bdxz\}$. 
Let $G_2$ be the same graph but labelled as shown at right in Figure \ref{fig:Quasigraphicframeliftexample1}, with $\mathcal{B}_2 = \{ abcd, \allowbreak wxyz\}$, $\mathcal{L}_2 = \{ bx, dz \}$, and $\Ff_2$ the set of unbalanced cycles not in $\Ll_2$. %$\mathcal{F}_2 = \{ aw, cy \}$. 
Then $F(G_1, \Bb_1) = M(G_2, \Bb_2, \Ll_2, \Ff_2)$. 
\begin{figure}[tbp] 
\begin{center} 
\includegraphics{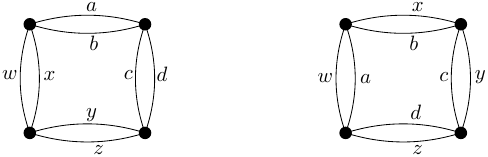}
\end{center} 
\caption{Example \ref{ex:differingreps}(b).}
\label{fig:Quasigraphicframeliftexample1} 
\end{figure} 
Let $\Bb_3 = \{abcd, wxyz, abyz, cdwx\}$. 
Then $L(G_1, \Bb_1) \allowbreak = \allowbreak M(G_2,\allowbreak \Bb_3, \Ll_2, \Ff_2)$. 
\end{example}

\section{On quasi-graphic matroids} 

In this section we apply Theorem \ref{fixed_thm:equivalences} to prove some results about quasi-graphic matroids. 

\subsection{Representability over a field} 

In \cite{JGT:JGT22177} Geelen, Gerards, and Whittle prove the following.

\begin{theorem}[{\cite[Theorem 1.4]{JGT:JGT22177}}] \label{T:GGWNonRepresentable}
Let $M$ be a 3-connected representable matroid. If $M$ is quasi-graphic, then $M$ is either a frame matroid or a lifted-graphic matroid.
\end{theorem}

An easy application of Theorem \ref{fixed_thm:equivalences} and Ingleton's inequality provides an alternate proof, and shows that 3-connectivity is not required. 
Ingleton showed \cite{Ingleton:TheInequality} that if $M$ is linearly representable over some field, then for all subsets $A, B, C, D \subseteq E(M)$ the following inequality holds: 
\begin{align*}
&r(A\cup B) + r(A \cup C) + r(A \cup D) + r(B \cup C) + r(B \cup D) \\
\ge & \ r(A) + r(B) + r(A \cup B \cup C) + r(A \cup B \cup D) + r(C \cup D)
\end{align*}
It is an easy check that a quasi-graphic matroid $M(G,\Bb,\Ll,\Ff)$ with neither $\Ll$ nor $\Ff$ degenerate violates Ingleton's inequality: 

\begin{theorem} \label{thm:nonrep}
Let $M$ be a 
{\color{black} connected} 
quasi-graphic matroid, and suppose $M$ is neither lifted-graphic nor frame. Then $M$ is not representable over any field.
\end{theorem}

\begin{proof}
Suppose for a contradiction that $M$ is linearly representable over a field. 
By Theorem \ref{fixed_thm:equivalences} $M = M(G, \Bb, \Ll, \Ff)$ where $G$ is a framework for $M$ and $(\Bb,\Ll,\Ff)$ is a proper tripartition of its cycles. 
Since $M$ is neither lifted-graphic nor frame, neither $\Ff$ nor $\Ll$ is degenerate. 
Hence there exist cycles $A, B, C, D$ where 
$A, B \in \Ll$ with $V(A) \cap V(B) = \emptyset$ 
and $C, D \in \Ff$ with $V(C) \cap V(D) = \emptyset$.
For notational convenience, set $a=|V(A)|$, $b=|V(B)|$, $c=|V(C)|$, $d=|V(D)|$, and ${ab} = |V(A) \cap V(B)|$, ${ac} = |V(A) \cap |V(C)|$, and so on. 
By Ingleton's inequality, and since every cycle in $\Ll$ meets every cycle in $\Ff$,
\begin{align*}
&r(A\cup B) + r(A \cup C) + r(A \cup D) + r(B \cup C) + r(B \cup D) \\
&=
(a+b-1) + (a+c-{ac}) + (a+d-{ad}) + (b + c - {bc}) + (b + d - {bd}) \\
\geq \
& r(A) + r(B) + r(A \cup B \cup C) + r(A \cup B \cup D) + r(C \cup D) \\
&= a+b+(a+b+c-ac-bc)+(a+b+d-ad-bd)+(c+d)
\end{align*}
which simplifies to the absurdity that $-1 \geq 0$. 
\end{proof}

\subsection{Minors}

In \cite{JGT:JGT22177} Geelen, Gerards, and Whittle prove the following.

\begin{theorem}[{\cite[Theorem 1.5]{JGT:JGT22177}}] \label{T:GGW_framework_with_loop}
Let $G$ be a framework for a 3-connected matroid $M$. If $G$ has a loop, then $M$ is either frame or lifted-graphic. 
\end{theorem}

An easy application of Theorem \ref{fixed_thm:equivalences} provides an alternative proof, for which 3-connectivity is not required. 
(We trust the reader is not confused by the common use of the term ``loop'' for a circuit of size 1 in a matroid as well as for an edge of a graph whose endpoints coincide.) 

\begin{theorem} \label{thm:quasigraphic_with_a_loop}
Let $G$ be a framework for a 
{\color{black} connected}
matroid $M$. 
If $G$ has a loop, then $M$ is either lifted-graphic or frame. 
\end{theorem}

\begin{proof}
By Theorem \ref{fixed_thm:equivalences}, $M = M(G,\Bb,\Ll,\Ff)$ for some proper tripartition $(\Bb,\Ll,\Ff)$ of the circuits of $G$. 
Let $e$ be a loop of $G$, incident to vertex $v$. 
Since $M$ does not have a loop, either $e \in \Ll$ or $e \in \Ff$. 
Suppose $e \in \Ll$. 
Since every cycle in $\Ll$ meets every cycle in $\Ff$, every cycle in $\Ff$ contains $v$. 
Thus $\Ff$ is degenerate, so $M$ is lifted-graphic. 
If $e \in \Ff$ then every cycle in $\Ll$ contains $v$, so $\Ll$ is degenerate and $M$ is frame. 
\end{proof}

{\color{black}
Disconnected counter-examples to Theorems \ref{thm:nonrep} and \ref{thm:quasigraphic_with_a_loop} are easy to construct. 
Let $N$ be a frame matroid that is not lifted-graphic, representable over a field $\FF$ with a framework $H$ with a loop, 
and let $N'$ be a lifted-graphic matroid that is not frame, representable over $\FF$ with a framework $H'$ that has a loop. 
Let $G$ be the disjoint union of $H$ and $H'$. 
Let $M$ be the direct sum of $N$ and $N'$. 
Then $G$ is clearly a framework for $M$, $M$ is clearly neither lifted-graphic nor frame, and $M$ is clearly representable over $\FF$. 
}

Now the proper tripartition construction of Theorem \ref{fixed_thm:equivalences} makes verifying that the class of quasi-graphic matroids is minor-closed a straightforward check. 
Let $G$ be a graph with proper tripartition $(\Bb,\Ll,\Ff)$ of its cycles. 
Let $(G,\Bb,\Ll,\Ff) \setminus e$ denote the graph $G \setminus e$ together with the tripartition $(\Bb',\Ll',\Ff')$ of the cycles of $G \setminus e$ obtained from $(\Bb,\Ll,\Ff)$ by taking
\begin{align*}
\Bb' = \{ C : C \in \Bb \text{ and } C \text{ does not contain } e \} \\
\Ll' = \{ C : C \in \Ll \text{ and } C \text{ does not contain } e \} \\
\Ff' = \{ C : C \in \Ff \text{ and } C \text{ does not contain } e \}
\end{align*}
As long as $e$ is not a loop, let $(G,\Bb,\Ll,\Ff)/e$ denote the graph $G/e$ together with the tripartition $(\Bb'',\Ll'',\Ff'')$ of the cycles of $G/e$ obtained from $(\Bb,\Ll,\Ff)$ by taking
\begin{align*}
\Bb'' = \{ C  : C \in \Bb \text{ or } C \cup e \in \Bb \} \\
\Ll'' = \{ C  : C \in \Ll \text{ or } C \cup e \in \Ll \} \\
\Ff'' = \{ C  : C \in \Ff \text{ or } C \cup e \in \Ff \}
\end{align*}
The following is a straightforward application of definitions and consideration of circuit-subgraphs. 

\begin{theorem} \label{prop:G2conn_minors}
\mbox{}
\begin{itemize}
\item $M(G,\Bb,\Ll,\Ff) \del e = M\br{(G,\Bb,\Ll,\Ff) \setminus e}$
\item $M(G,\Bb,\Ll,\Ff) / e = M\left((G,\Bb,\Ll,\Ff) / e \right)$ as long as $e$ is not a loop in $G$.
\end{itemize}
\end{theorem}
By Theorem \ref{thm:quasigraphic_with_a_loop} if $e \in E(G)$ is a loop and $e \notin \Bb$ then 
{\color{black} the component of} 
$M(G,\Bb,\Ll,\Ff)$ 
{\color{black} containing $e$} 
is lifted-graphic or frame and hence so is 
{\color{black} the corresponding component of} 
$M(G, \Bb, \Ll, \Ff)/e$. 
Thus minors of quasi-graphic matroids are again quasi-graphic.

\subsection{Connectivity} 

A connected lifted-graphic matroid may be represented by a disconnected graph. 
This cannot occur for quasi-graphic matroids that are not lifted-graphic. 

\begin{theorem} \label{thm:connected_QG_connected_G}
Let $G$ be a graph and let $(\Bb,\Ll,\Ff)$ be a proper tripartition of the cycles of $G$ such that $M(G,\Bb,\Ll,\Ff)$ is connected. 
\begin{enumerate} 
\item If $\Ff$ is non-degenerate then $G$ is connected. 
\item If neither $\Ff$ nor $\Ll$ is degenerate then $G$ is 2-connected. 
%\item If both $\Ff$ and $\Ll$ are non-degenerate then $G$ is 2-connected. 
\end{enumerate} 
\end{theorem}

\begin{proof} 
(1) 
Suppose to the contrary that $G$ has more than one component. 
Since $M(G,\Bb,\Ll,\Ff)$ is connected, it has a circuit consisting of a bracelet $C_1 \cup C_2$ with both cycles in $\Ll$, such that $C_1$ and $C_2$ are contained in different components of $G$. 
But a cycle in $\Ff$ cannot meet both $C_1$ and $C_2$, contradicting the fact that 
%$(\Bb,\Ll,\Ff)$ is proper. 
$\Ff$ is non-degenerate. 

(2) 
Suppose for a contradiction that both $\Ff$ and $\Ll$ are non-degenerate and that $G$ has a cut vertex $v$. 
Let $(A,B)$ be a partition of $E(G)$ with $V(A) \cap V(B) = \{v\}$. 
Each of $G[A]$ and $G[B]$ must contain an unbalanced cycle, else $M(G,\Bb,\Ll,\Ff)$ would be disconnected. 
Further, at least one of $G[A]$ or $G[B]$ must contain an unbalanced cycle avoiding $v$, since otherwise both $\Ff$ and $\Ll$ would be degenerate. 
Thus there is an unbalanced cycle $C$ in $G[A]$ and an unbalanced cycle $C'$ in $G[B]$ such that $C$ and $C'$ are disjoint. 
Suppose both $C$ and $C'$ are in $\Ff$. 
One of $C$ or $C'$ must contain $v$, for otherwise no cycle in $\Ll$ could meet both $C$ and $C'$, and $\Ll$ is non-empty by assumption. 
Without loss of generality, assume $C$ contains $v$. 
Then $C'$ does not contain $v$. 
This implies that no cycle in $G[A]$ is in $\Ll$, for such a cycle could not meet $C'$. 
Thus every cycle in $\Ll$ is in $G[B]$. 
But since every cycle in $\Ll$ must meet $C$, this implies that every cycle in $\Ll$ contains $v$ and so that $\Ll$ is degenerate, contrary to assumption. 
Similarly, the assumption that both $C$ and $C'$ are in $\Ll$ leads to the contradiction that $\Ff$ is degenerate. 
\end{proof}

Thus if $M$ is connected yet has a framework that is not connected, $M$ is lifted-graphic. 
Moreover, by Theorem \ref{thm:connected_QG_connected_G} if $M$ is connected we can always ask for a connected framework for $M$:

\begin{corollary} \label{thm:liftedgraphic_has_connected_graph} 
Let $G$ be a graph and let $(\Bb,\Ll,\Ff)$ be a proper tripartition of the cycles of $G$ such that $M=M(G,\Bb,\Ll,\Ff)$ is connected. 
\begin{enumerate} 
\item Either $G$ is connected or $M$ is lifted-graphic. 
\item If $M$ is lifted-graphic, then $M$ has a connected graph obtained by successively identifying pairs of vertices in distinct components of $G$. 
\end{enumerate}
\end{corollary}

Thus every 
{\color{black} connected}
quasi-graphic matroid has a connected framework. 

Our next result says that quasi-graphic matroids are 
% \st{essentially}
{in some sense, generally} 
3-connected: a connected quasi-graphic matroid that is neither lifted-graphic nor frame is either 3-connected or decomposes along 2-sums with graphic matroids. 

\begin{theorem} \label{lem:M2conn_G2conn}
Let $M$ be a connected matroid of the form $M(G, \Bb, \Ll, \Ff)$ with $(\Bb, \Ll, \Ff)$ a proper tripartition of $G$ such that neither $\Ll$ nor $\Ff$ is degenerate. 
If $M$ is not 3-connected, then $M$ is obtained via 2-sums of graphic matroids and a single 3-connected matroid of the form $M(H, \Bb', \Ll', \Ff')$, for some graph $H$ with proper tripartition $(\Bb',\Ll',\Ff')$. 
\end{theorem}

To prove Theorem \ref{lem:M2conn_G2conn}, we use the notion of a \emph{link-sum}. 
This is a 2-sum of graphs, together with an appropriate tripartition of its cycles, which provides a representation of the 2-sum of a quasi-graphic matroid and a graphic matroid.  
Let $G$ and $H$ be graphs and let $(\Bb,\Ll,\Ff)$ be a proper tripartition of the cycles of $G$. 
Assume that $E(G) \cap E(H) = \{e\}$ and that $e$ is a link in both $G$ and $H$. 
The 2-sum of the matroid $M(G,\Bb,\Ll,\Ff)$ and the cycle matroid $M(H)$ of $H$ on basepoint $e$ may be realised in the graphs as follows.
Let $u_1,v_1$ be the endpoints of $e$ in $G$ and let $u_2, v_2$ be the endpoints of $e$ in $H$. 
The \emph{link-sum} of $G$ and $H$ \emph{on} $e$ is the graph, denoted $G \oplus_2^e H$, obtained by identifying $u_1$ with $u_2$ and identifying $v_1$ with $v_2$, then deleting $e$, together with the following tripartition of its cycles. 
Let $\Bb'$ be the union of $\Bb$ and the set of all cycles of $H$ and the set 
\begin{multline*}
\{ P \cup Q : P \text{ is a $u_1$-$v_1$ path in $G$ with $P \cup e \in \Bb$} \\ \text{and $Q$ is a $u_2$-$v_2$ path in $H$} \}. 
\end{multline*}
Let $\Ll'$ be the union of $\Ll$ and the set 
\begin{multline*}
\{ P \cup Q : P \text{ is a $u_1$-$v_1$ path in $G$ with $P \cup e \in \Ll$} \\ \text{and $Q$ is a $u_2$-$v_2$ path in $H$} \} 
\end{multline*}
and let $\Ff'$ be the union of $\Ff$ and the set 
\begin{multline*}
\{ P \cup Q : P \text{ is a $u_1$-$v_1$ path in $G$ with $P \cup e \in \Ff$} \\ \text{and $Q$ is a $u_2$-$v_2$ path in $H$} \}.
\end{multline*}
It is straightforward to check that, regardless of the choices made for labelling the endpoints $u_1, v_1$ and $u_2, v_2$ of $e$, $(\Bb', \Ll', \Ff')$ is a proper tripartition of the cycles of $G \oplus_2^e H$, and that the 2-sum of $M(G,\Bb,\Ll,\Ff)$ and $M(H)$ on $e$ is equal to $M(G \oplus_2^e H, \Bb', \Ll', \Ff')$. 

\begin{proof}[Proof of Theorem \ref{lem:M2conn_G2conn}]
Suppose not for a contradiction, and let $M$ be a counterexample whose set of elements is minimal. In particular, there can be no 2-separation of $G$ which corresponds to a link-sum with a balanced biased graph. 
By Theorem \ref{thm:connected_QG_connected_G} $G$ is 2-connected. 
Since $M$ is connected but not 3-connected, it must have a 2-separation $(A, B)$. Let $c_A$ be the number of components of $G[A]$ and $c_B$ the number of components of $G[B]$. 
Choose $(A, B)$ so as to minimise $c_A + c_B$. Let $S$ be the set of vertices incident to edges in both $A$ and $B$. Since $G$ is 2-connected, no component of $A$ or $B$ can meet $S$ in fewer than 2 vertices.

\begin{claim}
There is no balanced component of $A$ or $B$ having precisely 2 vertices in $S$.
\end{claim}

\begin{proof}[Proof of Claim]
Suppose for a contradiction that there is such a component $X$, without loss of generality a component of $A$. 
Then $X$ consists of a single edge $x$ (else $M$ would be a 2-sum of a quasi-graphic matroid and a graphic matroid, given by the link-sum corresponding to the 2-separation $(X, E(G) \setminus X)$). 
But now $r(A \setminus \{x\}) = r(A) - 1$ and $r(B \cup \{x\}) \leq r(B) + 1$, so that the order of $(A \setminus \{x\}, B \cup \{x\})$ is at most that of $(A, B)$. Since the order of $(A, B)$ is at most 1 and that of $(A \setminus \{x\}, B \cup \{x\})$ is at least 1, we must have $r(B \cup \{x\}) = r(B) + 1$ and the order of $(A \setminus \{x\}, B \cup \{x\})$ is precisely 1. 

Since we chose $(A, B)$ to minimise $c_A + c_B$, $(A \setminus \{x\}, B \cup \{x\})$ cannot be a 2-separation and so $|A| = 2$. 
Since $M$ is connected and neither $\Ll$ nor $\Ff$ is degenerate, $G$ has no loops. Thus both elements of $A$ are links, so $|S|=4$. 
%Thus $|S|=4$. 
Since $r(B \cup \{x\}) = r(B) + 1$, $B$ is not spanning and so $A$ is codependent. Since $M$ is connected, $A$ must be a cocircuit, so its two elements are in series in $M$. But then $M$ is a 2-sum of $M/x$ with $M(K_3)$. $M/x$ is connected since $M$ is. By Proposition \ref{prop:G2conn_minors} $M/x = M((G, \Bb, \Ll, \Ff)/x)$. Our next aim is to show that neither $\Ll/x$ nor $\Ff/x$ is degenerate. 

There are two cases. 
The first is that at least two components of $B$ meet $X$. In this case, since each of these components meets $S$ in at least 2 vertices and $S$ has only 4 vertices, these two are the only components of $B$. Let $v$ be an endpoint of $x$. Since $\Ff$ is non-degenerate, there is some $C \in \Ff$ not containing $v$. Thus $C$ must be contained in some component $Y_1$ of $B$. Let $Y_2$ be the other component. $Y_2$ cannot also contain a cycle in $\Ff$, since $r(B) = r(M) - 1 = |V(B)| - 1$. But $Y_2$ cannot contain a cycle in $\Ll$ either, since $(\Bb, \Ll, \Ff)$ is proper. So $Y_2$ is balanced, and has precisely 2 vertices in $S$, so as above it consists of a single edge $y$. But then $x$ and $y$ are in series in $G$, and so neither $\Ll/x$ nor $\Ff/x$ can be degenerate.

The second case is that only one component $Y$ of $B$ meets $X$. Then $Y$ must contain at least one further vertex of $S$, and since $S$ has only 4 vertices $Y$ must contain $S$ and so in fact $B$ must be connected. In this case, since $r(B) = r(M) - 1 = |V(B)| - 1$ we must have that $B$ is balanced. Since $\Ff$ is not degenerate, there are disjoint cycles $C_1$ and $C_2$ in $\Ff$. Then $C_1$ and $C_2$ must meet $A$ in different edges, so one of them, say $C_1$, must contain $x$. Then $C_1/x$ and $C_2$ are disjoint cycles in $\Ff/x$, which is therefore also non-degenerate. Similarly $\Ll/x$ is non-degenerate.

By the minimality of $M$, $M/x$ can be obtained via 2-sums of graphic matroids and a single 3-connected matroid as in the statement of the theorem. 
Hence so can $M$, giving the desired contradiction. 
\end{proof}

We use $r_G$ to denote the rank function of $M(G)$ ($r$ is used for the rank function of $M$). We define $d$ to be $(r(A) - r_G(A)) + (r(B) - r_G(B))$. Then the equation for the order of $(A,B)$ tells us that $1 = r(A) + r(B) - r(E(G)) = (r_G(A) + r_G(B) - r_G(E(G))) - 1 + d \geq d$, since $G$ is 2-connected. So one of $r(A) - r_G(A)$ or $r(B) - r_G(B)$, without loss of generality the second, must be zero. Thus $B$ is balanced. Since every component of $A$ meets $S$ in at least 2 vertices we have $c_A \leq |S|/2$. Since by the claim every component of $B$ meets $S$ in at least 3 vertices we have $c_B \leq |S|/3$. 

Using the above formula for the order of $(A,B)$ once more, we have $1 = (r_G(A) + r_G(B) - r_G(E(G))) - 1 + d = |V(A)| - c_A + |V(B)| - c_B - |V(G)| + d = |S| - c_A - c_B + d \geq |S| - |S|/2 - |S|/3 + d = |S|/6 + d > d$. So $d = 0$ and thus $A$ is also balanced and by the above claim we have $c_A \leq |S| / 3$. Substituting this into the above calculation we have $1 = |S| - c_A  - c_B \geq |S| - |S|/3 - |S|/3 = |S|/3$, so $S$ contains at most 3 vertices. 
Since both $A$ and $B$ are balanced, every cycle in $\Ff$ must meet both $A$ and $B$ and so must contain at least 2 of these 3 vertices, contradicting the assumption that $\Ff$ is non-degenerate. 
\end{proof}

\section{Bases, independent sets, cocircuits} 

\subsection{Independent sets and bases} 
Recall that for a graph $G$ and a proper tripartition $(\Bb,\Ll,\Ff)$ of its cycles, the rank function of the matroid $M(G,\Bb,\Ll,\Ff)$ in terms of induced subgraphs is given by Lemma \ref{lem:rank}. 
Here we observe that this function may be expressed in terms of the rank functions of the lift matroid $L\GB$ and frame matroid $F\GB$. 
Let $\GB$ be a biased graph and let $X \subseteq E(G)$. 
The rank $r_L(X)$  of $X$ in the lift matroid $L\GB$ is given by \[r_L(X) = |V(X)| - c(X) + l(X)\] where $c(X)$ is the number of components of $G[X]$, and $l(X) = 0$ if $G[X]$ is balanced and $l(X)=1$ otherwise.
The rank $r_F(X)$ of $X$ in the frame matroid $F\GB$ is given by \[r_F(X) = |V(X)| - b(X)\] where $b(X)$ is the number of balanced components of $G[X]$.
Thus by Lemma \ref{lem:rank} the rank $r(X)$ of $X$ in $M(G,\Bb,\Ll,\Ff)$ is given by 
\begin{equation} \label{eqn:rankintermsofliftandframe}
r(X) = \begin{cases} r_F(X) &\text{ if } G[X] \text{ contains a cycle in } \Ff \\ r_L(X) &\text{ otherwise}. \end{cases}
\end{equation}

Thus $X \subseteq E(G)$ is independent in $M(G,\Bb,\Ll,\Ff)$ if $G[X]$ contains a cycle in $\Ff$ and is independent in $F\GB$ or if $X$ does not contain a cycle in $\Ff$ and is independent in $L\GB$. 
An explicit description is given in the following theorem. 

\begin{theorem} \label{P:IndependentSets}
Let $X \subseteq E(G)$. 
$X$ is independent in $M(G,\Bb,\Ll,\Ff)$ if and only if $X$ satisfies one of the following. 
\begin{enumerate}
\item $G[X]$ is a forest 
\item $G[X]$ contains just one cycle, which is in $\Ll$. 
\item  Each component of $G[X]$ contains at most one cycle and all cycles of $G[X]$ are in $\Ff$. 
\end{enumerate}
\end{theorem}

By Theorem \ref{P:IndependentSets} the bases of $M(G,\Bb,\Ll,\Ff)$ are edge sets of 
\begin{itemize} 
\item maximal subgraphs containing a single cycle that is in $\Ll$, and 
\item maximal subgraphs each of whose components is either a tree or contains a single cycle that is in $\Ff$. 
\end{itemize} 

\subsection{Cocircuits and a unique representation theorem}

We now determine the graphical structure of cocircuits and their complementary hyperplanes in quasi-graphic matroids. Since the structure of cocircuits in frame matroids and lifted-graphic matroids is understood, let us consider a proper tripartition $(\mc B,\mc L,\mc F)$ of the cycles of a graph $G$ 
with neither $\Ll$ nor $\Ff$ degenerate. 

For a biased graph $\GB$, a set of edges $X \subseteq E(G)$ is a \emph{balancing set} if $\Bb$ does not contain all cycles of $G$ but $\Bb$ contains all cycles of $G-X$. 
The rank function given in equation (\ref{eqn:rankintermsofliftandframe}), along with the fact that if neither $\Ll$ nor $\Ff$ is degenerate then $G$ has precisely one unbalanced component, immediately implies that 
the cocircuits of $M(G,\Bb,\Ll,\Ff)$ are precisely those edge sets $X$ minimal with respect to the property that either 
\begin{itemize} 
\item $X$ is a balancing set of $G$, 
\item $G-X$ has precisely one more balanced component than $G$ and contains a cycle in $\Ff$, or 
\item $G-X$ has precisely one more component than $G$ and a cycle in $\Ll$, but no cycle in $\Ff$.  
\end{itemize}
With only a little effort we obtain the following more precise description of the cocircuits of $M(G,\Bb,\Ll,\Ff)$. 
A depiction of the cocircuits described in the following theorem 
is given in Figure \ref{F:cocircuits}.

\begin{theorem} \label{thm:cocircuits} 
Let $G$ be a graph and let $(\Bb,\Ll,\Ff)$ be a proper tripartition of the cycles of $G$ with neither $\Ll$ nor $\Ff$ degenerate. 
The cocircuits of $M(G,\Bb,\Ll,\Ff)$ are precisely those edge sets $C$ satisfying one of the following.  
\begin{enumerate} 
\item $C$ is a minimal balancing set of $(G,\mc B)$. 
\item $C$ is a bond which separates $G$ into connected subgraphs $X$ and $Y$ for which $X$ is balanced and $Y$ is unbalanced.
\item $C$ is a bond which separates $G$ into connected subgraphs $X$ and $Y$ which are both unbalanced and every unbalanced cycle in $G-C$ is in $\mc L$.
\item $C$ is a disjoint union $K \cup B$ for which $G-K$ has connected components $X,Y_1,\ldots, Y_m$ with each edge in $K$ having exactly one endpoint in $X$, each $Y_i$ is unbalanced with each unbalanced cycle in $\mc F$, and $B$ is a minimal balancing set (possibly empty) of $X$.
\end{enumerate}
\end{theorem}

A \emph{star} in a graph is the set of edges that are not loops incident to a vertex. 
The star whose edges are incident to the vertex $v$ is the star \emph{at} $v$. 
A vertex $v$ is a \emph{balancing vertex} if its set of incident edges is a balancing set. 
Let us call a proper tripartition $(\Bb,\Ll,\Ff)$ of the cycles of a graph $G$ \emph{degenerate} if one of $\Ll$ or $\Ff$ is degenerate. 
Observe that if $G$ has a balancing vertex, then necessarily $(\Bb,\Ll,\Ff)$ is degenerate. 

\begin{proposition} \label{prop:star_cocircuits}
Let $G$ be a graph with non-degenerate proper tripartition $(\Bb,\Ll,\Ff)$ such that $M(G,\mc B,\mc L,\mc F)$ is connected. Then for each vertex $v$ of $G$, the star at $v$ is a cocircuit of $M(G,\Bb,\Ll,\Ff)$. 
\end{proposition}

\begin{proof}
Since the tripartition is non-degenerate and $M(G,\Bb,\Ll,\Ff)$ is connected, $G$ is loopless. 
By Theorem \ref{thm:connected_QG_connected_G}, $G$ is 2-connected. 
Thus the set of edges $X$ incident to a vertex $v$ separates $G$ into exactly two components: the isolated vertex $v$ and $G-v$. 
Because $G$ has no balancing vertex, $G-v$ is unbalanced. 
Hence $E(G)-X$ has rank one less than the rank of $E(G)$ and is minimal with respect to this property. 
The result follows.  
\end{proof}

\begin{figure}
\begin{center}
\includegraphics[page=6,scale=.8]{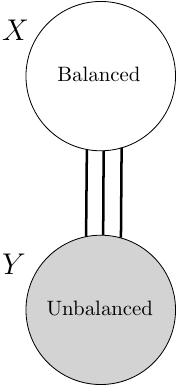}\hspace{1cm}
\includegraphics[page=7,scale=.8]{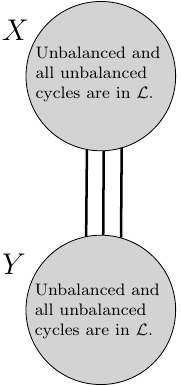}

\vspace{.5cm}

\includegraphics[page=8,scale=.7]{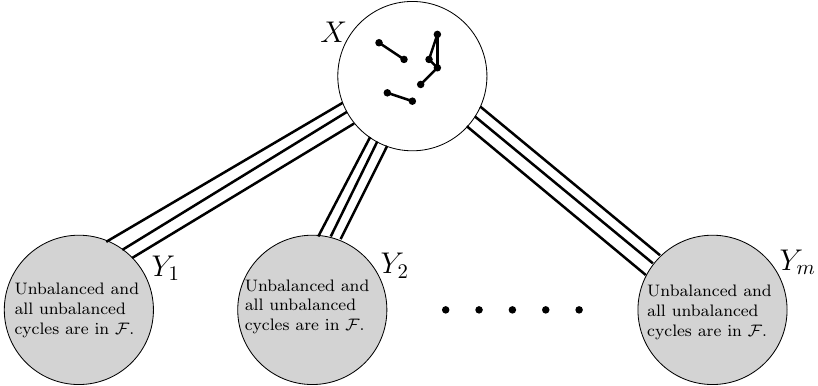}
\end{center}
\caption{Depictions of cocircuits from Parts (2), (3), and (4) of Theorem \ref{thm:cocircuits}.}
\label{F:cocircuits}
\end{figure}

\begin{proof}[Proof of Theorem \ref{thm:cocircuits}]
Let $M=M(G,\Bb,\Ll,\Ff)$. 
By Corollary \ref{thm:liftedgraphic_has_connected_graph} $G$ is connected, so the rank of $M$ is $|V(G)|$. 
Let $C$ be a subset of $E(G)$ of the form described in one of statements (1)-(4) in the theorem. 
Then $E(G)-C$ is closed and has rank $r(M)-1$, so $C$ is a cocircuit of $M$. 

For the converse, let $C$ be a cocircuit of $M$. 
Assume that $C$ does not contain a balancing set. 
Then $G-C$ contains an unbalanced cycle and has rank less than that of $M$, so the rank function implies that $G-C$ is disconnected. 
Let $X,Y_1,\ldots,Y_m$ be the components of $G-C$. 
If all of $X,Y_1,\ldots,Y_m$ are unbalanced, then all of the unbalanced cycles in $G-C$ are either in $\mc L$ or in $\mc F$. 
In the latter case, the rank of $G-C$ is still $|V(G)|$, a contradiction. 
Thus all of the unbalanced cycles in $G-C$ are in $\mc L$ and by minimality $C$ has the form described in statement (3). 
So now assume that not all of $X,Y_1,\ldots,Y_m$ are unbalanced; suppose $X$ is balanced. 
The rank function implies that $X$ is the only balanced component. 
Let $K$ be the subset of $C$ consisting of edges whose endpoints are in distinct components of $G-C$ and let $B$ be the subset of edges of $C$ consisting of edges with both endpoints in the same component of $G-C$. 
If $m=1$, then the minimality of $C$ implies that either $B=\emptyset$ and $C$ has the form of statement (2) or $B\neq\emptyset$ and $C$ the form of statement (4). 
So now assume $m \geq 2$. 
As before, either all unbalanced cycles remaining in $G-C$ are in $\Ll$ or all are in $\Ff$. 
The rank function implies that it must be the latter: all unbalanced cycles in $Y_1,\ldots,Y_m$ are in $\Ff$. 
Finally, the minimality of $C$ implies that each edge in $K$ has one endpoint in $X$ and that each edge in $B$ has both endpoints in $X$. That is, $C$ has the form described in statement (4).
\end{proof}

The characterisation of the cocircuits of $M(G,\Bb,\Ll,\Ff)$ given in Theorem \ref{thm:cocircuits} yields the following sufficient conditions for uniqueness of representation for quasi-graphic matroids. 

\begin{theorem} \label{thm:unique_representation}
Let $G$ be a 4-connected graph and let $(\Bb,\Ll,\Ff)$ be a proper tripartition of the cycles of $G$. 
Assume that for each vertex $v$ of $G$ the tripartition of the cycles of $G-v$ induced by $(\Bb,\Ll,\Ff)$ is non-degenerate. 
Then $G$ is the unique framework for $M(G,\Bb,\Ll,\Ff)$. 
\end{theorem}

We say a hyperplane $H$ of a matroid $M$ is 3-connected or binary when the restriction of $M$ to $H$ has the property of being 3-connected or binary. 

\begin{lemma} \label{lem:starcocircuits}
Let $M$ be a matroid with framework $G$. 
The complementary cocircuit of a 3-connected non-binary hyperplane of $M$ is a star. 
\end{lemma}

\begin{proof} 
Let $(\Bb,\Ll,\Ff)$ be a proper tripartition of the cycles of $G$ such that $M=M(G,\Bb,\Ll,\Ff)$. 
If $(\Bb,\Ll,\Ff)$ is degenerate then $M$ is either frame or lifted-graphic, so the statement holds. 
So assume $(\Bb,\Ll,\Ff)$ is non-degenerate. 
Let $C$ be the complementary cocircuit of a 3-connected non-binary hyperplane $H$ of $M$. 
Then $C$ has one of the forms described in statements (1)-(4) of Theorem \ref{thm:cocircuits}. 
\begin{itemize} 
\item $C$ is not (1) a minimal balancing set of $\GB$, since then $H$ would be graphic. 
\item $C$ is not (2) a bond separating $G$ into connected subgraphs $X$ and $Y$ for which $X$ is balanced and $Y$ is unbalanced, unless $E(X)=\emptyset$, since otherwise $H$ would be disconnected. 
\item $C$ is not (3) a bond separating $G$ into two connected subgraphs $X$, $Y$, both of which are unbalanced and with every unbalanced cycle in $\Ll$, since then $(E(X), E(Y))$ would be a 2-separation of $H$. 
\item $C$ is not of the form of statement (4), since then $H$ would be disconnected. 
\end{itemize} 
Thus $C$ is a star. 
\end{proof}

\begin{lemma} \label{lem:aU24minor}
Let $G$ be a 2-connected graph with proper tripartition $(\Bb,\Ll,\Ff)$ of its cycles. 
If $\Ff$ is non-degenerate, then $M(G,\Bb,\Ll,\Ff)$ is non-binary. 
\end{lemma}

\begin{proof} 
Let $C$ and $C'$ be a pair of vertex-disjoint cycles in $\Ff$, and let $P, P'$ be a pair of vertex-disjoint paths linking $C$ and $C'$. 
By \cite[Lemma 6]{Slilaty:UniqueRepresentations}, the restriction of $M(G,\Bb,\Ll,\Ff)$ to $E(C \cup P \cup P' \cup C')$ contains $U_{2,4}$ as a minor. 
\end{proof}

\begin{proof}[Proof of Theorem \ref{thm:unique_representation}]
Set $M=M(G,\Bb,\Ll,\Ff)$. 
Let $v \in V(G)$. 
By Proposition \ref{prop:star_cocircuits} the star at $v$ is a cocircuit of $M$. 
Let $H$ be its complementary hyperplane. 
Since in the induced tripartition of the cycles of $G-v$ the collection $\Ff$ is non-degenerate, by Lemma \ref{lem:aU24minor} $H$ is non-binary. 
Since $G$ is 4-connected, $G-v$ is 3-connected. 
We claim $H$ is 3-connected. 
To see that this is so, suppose the contrary. 
Then by Theorem \ref{lem:M2conn_G2conn} $H$ is obtained via 2-sums of a single 3-connected quasi-graphic matroid with graphic matroids. 
But any such 2-sum induces either a 1- or a 2-separation of $G[E(H)]$, a contradiction. 

Thus $M$ has $|V(G)|$ 3-connected non-binary hyperplanes. 
By Lemma \ref{lem:starcocircuits}, in every framework for $M$ the complementary cocircuit of each of these hyperplanes is a star. 
\end{proof}

\section{Biased-graphic matroids} 
\label{sec:Biased_graphic_matroids}

We now return to Zaslavsky's question about matroids whose independent sets are ``intermediate" between those of $L\GB$ and $F\GB$ for a given biased graph $\GB$. 
As noted at the end of Section \ref{sec:quasigraphicintro}, if $M$ is quasi-graphic and $G$ is a framework for $M$, then setting $\Bb = \{ C : C$ is a cycle of $G$ and a circuit of $M\}$ yields a biased graph $\GB$ {\color{black} for which $M$} satisfies Zaslavsky's intermediate condition  
\[
\Ii(L\GB) \subseteq \Ii(M) \subseteq \Ii(F\GB). 
\]
{\color{black}
If $M$ or $G$ is connected, then (by Theorems \ref{thm:BraceletFunctionPropriety} and \ref{new_Thm_2point1}) defining a bracelet function $\chi$ on the set of bracelets of $(G,\Bb)$ by $\chi(B) = \mathsf{dependent}$ if and only if $B$ is a circuit of $M$, yields the precise description of $M$ as $M(G,\Bb,\chi)$, and the obvious statement 
\[\Ii(L\GB) \subseteq \Ii(M(G,\Bb,\chi)) \subseteq \Ii(F\GB) \]
since $L(G,\Bb)$ and $F(G,\Bb)$ are the matroids $M(G,\Bb,\chi')$ with $\chi'$ mapping every bracelet to $\mathsf{dependent}$, and $\mathsf{independent}$, respectively.

Zaslavsky has asked for an intermediate matroid construction that respects restriction. 
By Theorem \ref{fixed_2_thm:equivalences}, the map given by bracelet functions on biased graphs 
from the set of all connected biased graphs to the set of matroids is an intermediate matroid construction; by Theorem \ref{prop:G2conn_minors} this map respects restriction. 
% What if neither $M$ nor $G$ is connected? 
But the domain of this map is restricted to connected graphs. We would like an intermediate construction with domain all biased graphs. 
}

\subsection{$I$-component-respecting Whitney equivalent graphs} \label{sec_component-resp-Whitney-equiv}
{\color{black}
This section deals with some technicalities regarding descriptions of quasi-graphic matroids that are not connected. 
The definition of a quasi-graphic matroid clearly includes disconnected matroids, but our bracelet function and tripartition constructions in general do not. 
Here we address this issue. 
In the next subsection, in answer to Zaslavsky's question we introduce a new class of matroids which we call \emph{biased-graphic}. 
The notions introduced in this subsection on the interplay between matroid and graph connectivity motivate the key part of the definition of this class. 

Let $G$ be a graph, and let $H_1, \ldots, H_n$ be the connected components of $G$. 
Let $I = \{I_1, I_2, \ldots, I_m\}$ be a partition of $\{1, 2, \ldots, n\}$; 
for each $I_j \in I$, let $G_{I_j}$ be the disjoint union $\bigcup_{i \in I_j} G_i$. 
For each subgraph 
$G_{I_j}$ of $G$, 
let $\Bb_{I_j}$ be a collection of cycles of $G_{I_j}$ satisfying the theta property, 
and let $\chi_{I_j}$ be a proper bracelet function defined on the bracelets of $(G_{I_j}, \Bb_{I_j})$. 
Let $\Bb = \bigcup \Bb_{I_j}$ and let $\chi_I$ be the bracelet function assigning $\mathsf{dependent}$ to a bracelet $B$ of $G$ if and only if $\chi_{I_j}(B) = \mathsf{dependent}$ for some bracelet function $\chi_{I_j}$, $I_j \in I$. 
While $\chi_I$ is not necessarily proper, the restriction of $\chi_I$ to each biased subgraph $(G_{I_j}, \Bb_{I_j})$ of $(G,\Bb)$ is proper; 
we say in this case that $\chi_I$ is \emph{component-wise proper}. 

It is straightforward to see that proof of Theorem \ref{thm:bracelet_defn_of_matroid} extends to component-wise proper bracelet functions, as follows.

\begin{theorem} \label{thm:component-wise-proper_bracelet_function_defn}
Let $(G,\Bb)$ be a biased graph, and let $\chi_I$ be a component-wise-proper bracelet function for $(G,\Bb)$ for some partition $I$ of the set indexing the connected components of $G$. 
Then $\Cc(G,\Bb,\chi_I)$ is the set of circuits of a matroid. 
\end{theorem}

As before, we denote the matroid of Theorem \ref{thm:component-wise-proper_bracelet_function_defn} by $M(G,\Bb,\chi_I)$. 

Given a biased graph $(G,\Bb)$ with connected components $G_1, \ldots, G_n$ along with a partition $I$ of $\{1, \ldots, n\}$ and a component-wise proper bracelet function $\chi_I$ for $(G,\Bb)$ with respect to $I$, 
{\color{black}for each $I_j$ of the partition, let $H_{I_j}$ be the union of those connected components of $G_{I_j}$ that contain an unbalanced cycle contained in a bracelet assigned $\mathsf{dependent}$ by $\chi_{I_j}$, and let $H_{I_j}'$ be a graph obtained from $H_{I_j}$}
by successively identifying pairs of vertices in distinct connected components of 
{\color{black}$H_{I_j}$.} 
{\color{black}
Let $G'$ be the disjoint union of the graphs $H_j'$. 
Then $H_{I_j}'$ is a connected component of $G'$ if and only if $H_{I_j}'$ is obtained from the set of graphs $\{G_i : G_i \subseteq H_{I_j}\}$. 
Now let $H$ be the disjoint union of all connected components $G_i$ of $G$ not contained in any subgraph $H_{I_j}$ along with all the connected subgraphs $H_{I_j}'$. 
It is clear that for each $I_j$ the restriction of $M(G,\Bb,\chi_I)$ to $E(H_{I_j})$ is a connected matroid. 
Hence by Corollary \ref{thm:liftedgraphic_has_connected_graph}, $H$ is a framework for $M(G,\Bb,\chi_I)$.
}
Let us call $H$ an \emph{$I$-component-respecting} Whitney equivalent graph for $G$ and $\chi_I$.
{\color{black}We record this fact in the following theorem.} 

\begin{theorem} \label{thm:component-wise-proper_bracelet_function_defn_WhitneyEq_version}
Let $(G,\Bb)$ be a biased graph, and let $\chi_I$ be a component-wise-proper bracelet function for $(G,\Bb)$ for some partition $I$ of the set indexing the connected components of $G$. 
Let $H$ be an $I$-component-respecting Whitney equivalent graph for $G$. 
Then $\Cc(H,\Bb,\chi_I)$ is the set of circuits of the matroid $M(G,\Bb,\chi_I)$, 
and each component of $M(G,\Bb,\chi_I)$ is contained in a component of $H$. 
\end{theorem}

Conversely, given a matroid $M$ with framework $G$, 
if $M$ has components $M_1, \ldots, M_n$ while $G$ has connected components $G_1, \ldots, G_m$, 
then by Corollary \ref{thm:liftedgraphic_has_connected_graph} $M$ has an $I$-component-respecting Whitney equivalent framework $H$ for $G$, where 
\[ I = \{\{i : E(G_i) \subseteq M_j \} : j \in \{1, \ldots, n\}\}.\] 
Taking as $\Bb$ those cycles of $G$ that are circuits of $M$ and defining $\chi_I$ as the bracelet function whose dependent bracelets are precisely those given by a proper bracelet function $\chi_j$ for each component $H_j$ of $H$ defined according to dependence in component $M_j$ of $M$, we obtain a biased graph $(H,\Bb)$ along with a component-wise-proper bracelet function $\chi_I$ for which $M = M(H,\Bb,\chi_I)$. 
We may define a bracelet function $\chi_I'$ for $\GB$ by $\chi_I'(B) = \chi_I(B)$ if $B$ is a bracelet of both $G$ and $H$, and $\chi_I'(B) = \mathsf{dependent}$ if $B$ is a bracelet of $G$ that is a tight handcuff in $H$. 
Let us call a bracelet function $\chi_I'$ that may be be obtained in this manner a 
\emph{handcuff-linking component-wise proper bracelet function}. 
Then for every subset $X$ of $E(M)$, $X \in \Cc(G,\Bb,\chi_I')$ if and only if $X \in \Cc(H,\Bb,\chi_I)$. 
Thus $M = M(G,\Bb,\chi_I') = M(H,\Bb,\chi_I)$, and 
\[\Ii(L\GB) \subseteq \Ii(M) \subseteq \Ii(F\GB) \]
The notions of component-wise proper and 
handcuff-linking component-wise proper bracelet functions, 
Theorem \ref{thm:component-wise-proper_bracelet_function_defn}, along with Theorem \ref{fixed_thm:equivalences}, extend the bracelet function and tripartition descriptions to matroids that are not necessarily connected. 
For any biased graph $\GB$ with 
component-wise proper 
bracelet function $\chi_I$
with respect to some partition $I$ of the set indexing the connected components of $G$, 
$M(G,\Bb,\chi_I)$ is a quasi-graphic matroid. 
We now have an intermediate matroid construction respecting restriction whose domain includes all biased graphs equipped with component-wise bracelet functions: 
given a biased graph $(G,\Bb)$, a partition $I$ of the index set of its connected components, and a corresponding 
component-wise proper 
bracelet function $\chi_I$, 
\[ \Ii(L(G,\Bb)) \subseteq \Ii(M(G,\Bb,\chi_I)) \subseteq \Ii(F(G,\Bb)).\] 
}

Given a biased graph $\GB$, are there any other matroids intermediate between $L\GB$ and $F\GB$? 
We propose a natural 
{\color{black}(and necessary)} 
non-degeneracy condition, and show that subject to this condition the answer is, 
{\color{black}``yes, but not that are 3-connected''.}
More precisely, given a biased graph $\GB$, subject to this non-degeneracy condition we show that if $G$ is 2-connected then all matroids intermediate for $\GB$ are quasi-graphic, and if $G$ is connected but not 2-connected, then all matroids intermediate for $\GB$ are obtained as 2-sums of lifted-graphic and frame matroids.  
We also show that if $G$ is a framework for a connected quasi-graphic matroid $M$ and $G$ is not 2-connected, then $M$ is either lifted-graphic or frame.  
Thus we show that the only matroids intermediate between $L\GB$ and $F\GB$ for some biased graph $\GB$ that are not quasi-graphic arise as 2-sums of lifted-graphic and frame matroids, while all quasi-graphic matroids with a framework that is not 2-connected are obtained as 2-sums in which either all summands are lifted-graphic or all summands are frame. 

\subsection{Biased-graphic matroids} \label{sec_BGMatroids}

Let $M$ be a matroid. 
Let us say that $M$ is \emph{biased-graphic}
if 
{\color{black}there exists a graph $G$ for which}
each of the following hold:
\begin{enumerate}
\item[{\scshape (b1)}] \hypertarget{BGM(1)} $E(G) = E(M)$, 
\item[{\scshape (b2)}] \hypertarget{BGM(2)} the collection $\Bb = \{ C : C$ is a circuit of $M$ and a cycle of $G\}$ satisfies the theta property and
\[
\Ii(L\br{G,\Bb}) \subseteq \Ii\br{M} \subseteq \Ii\br{F(G,\Bb)}
\]
\item[{\scshape (b3)}] \hypertarget{BGM(3)} every component of $M$ is contained in a component of $G$.
\end{enumerate} 

{\color{black}
A \emph{vertex identification} operation in a graph is the operation of identifying to a single vertex a pair of distinct vertices, one from each of two distinct components. 
We say a graph $H$ is obtained from $G$ \emph{by vertex identification} if $H$ is obtained from $G$ by a sequence vertex identification operations.
Note that if $G$ is a graph satisfying 
\BGone\ and \BGtwo, 
and $H$ is obtained from $G$ by vertex identification, then $G$ and $H$ have precisely the same set of cycles. 
Thus the collection of cycles of $H$ that are circuits of $M$ and the collection of cycles of $G$ that are circuits of $M$ are precisely the same. 

{\color{black}
\begin{obs} \label{relation_bt_vertex_identified_graphs}
Let $M$ be a matroid and let $H$ be a graph satisfying 
\BGone, \BGtwo, and \BGthree. 
If $H$ is obtained from the graph $G$ by vertex identification, then $G$ satisfies 
\BGone\ and \BGtwo. 
\end{obs}

\begin{proof} 
The matroids $L(H,\Bb)$ and $L(G,\Bb)$ are equal. 
All circuits of $F(G,\Bb)$ induce connected subgraphs of $(G,\Bb)$, so if $X$ is independent in $F(H,\Bb)$, then $X$ is independent in $F(G,\Bb)$. 
\end{proof}
}

This observation justifies the following definition. 
Let $M$ be a biased graphic matroid. 
A graph $G$ is a \emph{graph for $M$} if there is a graph $H$ obtained from $G$ by vertex identification that satisfies 
\BGone, \BGtwo, and \BGthree. 
We also say that $M$ is \emph{represented by $G$}, that $G$ is a \emph{representation for} $M$, and that $M$ is biased-graphic \emph{with graph} $G$. 

Let us provide a brief description of the rationale for this definition. 
A disconnected biased graph $(G,\Bb)$ may have its lift matroid $L(G,\Bb)$ connected. 
We would like to allow such representations. 
Moreover, we would like the restriction of a biased-graphic matroid $M$, say represented by the graph $G$, to be represented by the corresponding subgraph of $G$. 
That is, if $N = M \del X$, and $G$ is a graph for $M$, then we would like that $G \del X$ provide a representation for $N$. 
It turns out that while condition 
\BGthree\
in the definition of a biased-graphic matroid is crucial, it need not prevent us from allowing disconnected graphs to represent a connected matroid. 
Before we can prove that this is in fact the case, we need to know more about biased-graphic matroids and graphs satisfying 
\BGone, \BGtwo, and \BGthree. 

When $G$ is a graph for the matroid $M$, we denote by $\Bb_G$ the set of cycles of $G$ that are circuits of $M$, and by $H_G$ a graph obtained from $G$ by vertex identification that satisfies 
\BGone, \BGtwo, and \BGthree. 
As before, we refer to the cycles in the collection $\Bb_G$ as \emph{balanced}. 
We write $(G,\Bb_G,H_G)_M$, or simply $(G,\Bb,H)$ when $M$ is clear from context, for the triple consisting of the graph $G$ for the matroid $M$ with collection of balanced cycles $\Bb_G$ and graph $H$ obtained from $G$ by vertex identification for which 
\BGone, \BGtwo, and \BGthree\
hold. 
We say in this case that the triple $(G,\Bb,H)$ 
provides 
\emph{biased-graphic credentials} 
for $M$, and that $(G,\Bb,H)$ 
\emph{are credentials} for $M$. 

Condition \BGthree\ is required to deal with the unavoidable discrepancies that exist in graph versus matroid connectivity. 
The examples of graphic, frame, and lifted graphic matroids show that we cannot demand that the connectivity of a graph $G$ and of a matroid defined on the edge set of $G$ correspond. 
Graphs Whitney equivalent to a given graph $G$ have precisely the same set of cycles as $G$, so 
we ask instead just that there exist a Whitney equivalent graph whose components respect those of $M$, in the sense of condition 
\BGthree. 
In fact, we do not require Whitney flips, just Whitney's vertex identification and cleaving (the inverse of identification)  operations.}

Thus $M$ is biased-graphic if there is a biased graph $\GB$ such that $M$ is intermediate between $L\GB$ and $F\GB$, so long as 
{\color{black}there exists a graph $H$ Whitney equivalent to $G$, for which} 
distinct components of $M$ do not meet more than one component of {\color{black}$H$, and for which $M$ remains intermediate between $L(H,\Bb)$ and $F(H,\Bb)$.}
Condition \BGthree\ is a non-degeneracy condition required to avoid issues such as that raised by Example \ref{ex:bad_all_loops_example} in Section \ref{sec:Intermediate_matroids}. 

All circuits of a frame matroid $F\GB$ induce connected subgraphs of $G$, so all frame matroids are biased-graphic. 
{\color{black}Our non-degeneracy condition 
\BGthree\
is inspired by the fact} that while there exist biased graphs $\GB$ in which $G$ is disconnected while the lifted-graphic matroid $L\GB$ is connected, in this case there always exists a connected graph $H$ obtained by identifying pairs of vertices in distinct components of $G$ such that $L\GB = L(H,\Bb)$ (by Corollary \ref{thm:liftedgraphic_has_connected_graph}). 
Thus condition 
\BGthree\
does not exclude any lifted-graphic matroid. 
{\color{black}In fact, Corollary \ref{thm:liftedgraphic_has_connected_graph} implies that all quasi-graphic matroids are biased graphic.}

\begin{theorem} \label{thm:allquasigraphicarebiasedgraphic}
Let $M$ be a quasi-graphic matroid.
Then $M$ is biased-graphic.
\end{theorem}

\begin{proof}
It suffices to show that every connected quasi-graphic matroid is biased-graphic
{\color{black}(as we may apply the component-wise proper bracelet function construction separately to each component of a given matroid)}.
Let $M$ be a connected quasi-graphic matroid.
{\color{black} 
By Theorem \ref{fixed_thm:equivalences}, $M = M(G,\Bb,\Ll,\Ff)$ for some graph $G$ with proper tripartition $(\Bb,\Ll,\Ff)$ of its cycles.}
{\color{black} So \BGone\ and \BGtwo\ are} satisfied. 
{\color{black}By Corollary \ref{thm:liftedgraphic_has_connected_graph} $M$ has a connected framework $H$ 
obtained by vertex identification from $G$.}
Since $M$ and $H$ are both connected 
{\color{black}and $M=M(H,\Bb,\Ll,\Ff)$}, 
\BGthree\ holds. 
\end{proof}

Are there biased-graphic matroids that are not quasi-graphic? 
The answer is yes, but rather surprisingly, none that are 3-connected. 
Moreover there is an easy characterisation in terms of 2-sums of lifted-graphic and frame matroids. 
To show this, we now turn our attention to the class of biased-graphic matroids itself. 

Recall that the cyclomatic number 
{\color{black}$\beta_G(X)$, or just}
$\beta(X)$ 
{\color{black}when the graph $G$ is clear in context,}
of a subset $X \subseteq E(G)$ is the minimum number of edges that must be removed from the induced subgraph $G[X]$ in order to obtain an acyclic subgraph. 
We say $X$ is \emph{unicyclic} if $\beta(X)=1$. 
The following lemma is key. 
The class of quasi-graphic matroids requires condition \QGfour\ in the definition of a framework in order to avoid the ``catastrophe'' that without it, all matroids would be quasi-graphic \cite{JGT:JGT22177}. 
Though it is a weaker statement than that of condition 
\QGfour\
for frameworks, Lemma \ref{lem:a_circuit_has_at_most_two_components} provides the necessary structure to ensure our class is meaningful. 
{\color{black}It is so that we may deduce Lemma \ref{lem:a_circuit_has_at_most_two_components} that axiom 
\BGthree\
is required in the definition of biased-graphic matroid. 
In fact, just as without axiom 
\QGfour\
for quasi-graphic matroids every matroid would be quasi-graphic, without axiom 
\BGthree\
every matroid would be biased-graphic.}

\begin{lemma} \label{lem:a_circuit_has_at_most_two_components}
Let $M$ be a biased-graphic matroid and let $G$ be a graph for $M$.
If $X$ is a circuit of $M$ then $\beta(X) \leq 2$.
\end{lemma}

\begin{proof}
Let $X$ be a circuit of $M$ and suppose for a contradiction that $\beta_G(X)>2$.
{\color{black}
Let $H$ be a graph Whitney equivalent to $G$ with every component of $M$ contained in a component of $H$, and $\Ii(M) \subseteq \Ii(F(H,\Bb))$.} 
{\color{black}
Then $\beta_H(X) = \beta_G(X)$.}
Then 
{\color{black}$H[X]$}
does not contain a cycle in $\Bb$, and nor does 
{\color{black}$H[X]$}
contain a theta, a pair of tight handcuffs, or a pair of loose handcuffs: each of these has cyclomatic number $\leq 2$ and since each is dependent in $F\GB$, this would imply that $X$ strictly contains a dependent set, a contradiction.
Thus 
{\color{black}$H[X]$}
has at least three components and each component of 
{\color{black}$H[X]$}
has cyclomatic number at most 1.

Let $X_1,\ldots, X_n$ be the edge sets of the components of 
{\color{black}$H[X]$}
with $\beta(X_i)=1$, and let $T_1, \ldots, T_m$ be the edge sets of components of 
{\color{black}$H[X]$}
with $\beta(T_j)=0$.
By assumption $n \geq 3$.
For each $i \in \{1, \ldots, n\}$, let $C_i$ be the cycle contained in $X_i$ and let $e_i$ be an edge in $C_i$.
Then 
{\color{black}$H[X] \del \{e_1, \ldots, e_n\}$}
is a forest.
Since 
{\color{black}$H$}
is connected,
there is a set of edges $Y$ such that the induced subgraph $H[X] \del \{e_1, \ldots, e_n\} \cup Y$ is a tree. 
Let $Z_1 = X \del \{e_2, \ldots, e_n\} \cup Y$.
Since $Z_1$ contains just one cycle (the cycle $C_1$) which is unbalanced, $Z_1$ is independent in $L\GB$, and so independent in $M$.
Consider $Z_1 \cup e_2$. 
This set contains the loose handcuff $H_2$ 
consisting of $C_1 \cup C_2$ together with the path $P_2$ in $G[X] \del \{e_1, \ldots, e_n\} \cup Y$ linking $C_1$ and $C_2$. 
Since $H_2$ is a circuit of $F(H,\Bb)$, $H_2$ is dependent in $M$. 
{\color{black}Because $H_2-e_2$ is independent in $L(H,\Bb)$, $H_2-e_2$ is independent in $M$, so} 
more precisely, $Z_1 \cup e_2$ contains a unique circuit $W_2 \subseteq H_2$ {\color{black}of $M$}.
Since $W_2$ is not contained in $X$, $W_2$ contains at least one edge $f_2$ that is not contained in $X$. 
Thus $f_2 \in P_2$. 
Let $Z_2 = Z_1 \cup e_2 - f_2$.
Since the deletion of $f_2$ destroys the circuit $W_2$, and $W_2$ is the only circuit contained in $Z_1 \cup e_2$, $Z_2$ is independent.
Observe that $H[Z_2]$ has two unicyclic components, one of which contains $C_1$, the other containing $C_2$, and that for each $i \in \{3, \ldots, n\}$, $C_i - e_i$ is contained in one of these two components.

Now consider $Z_2 \cup e_3$.
Let $C' \in \{C_1, C_2\}$ be the cycle disjoint from but contained in the same component of $H[Z_2 \cup e_3]$ as $C_3$. 
Let $P_3$ be the unique path in $Z_2$ linking $C'$ and $C_3$. 
Then $H_3 = C' \cup P_3 \cup C_3$ is a loose handcuff contained in $Z_2 \cup e_3$. 
As before, since $H_3$ is dependent in $F(H,\Bb)$, $H_3$ is dependent in $M$ and so $M$ has a 
{\color{black}unique}
circuit $W_3 \subseteq H_3$. 
Since $W_3$ is not contained in $X$, $W_3$ contains an edge $f_3 \in P_3$ that is not contained in $X$.
Let $Z_3 = Z_2 \cup e_3 - f_3$.
Since deleting $f_3$ destroys the circuit $W_3$ and $Z_2 \cup e_3$ does not contain any other circuit, $Z_3$ is independent.
Observe that $H[Z_3]$ has three unicyclic components, which contain cycles $C_1$, $C_2$, and $C_3$, respectively, and that for each $i \in \{4, \ldots, n\}$, $C_i - e_i$ is contained in one of these components.

Continue in this manner.
In each step $i$ we add edge $e_{i+1}$ to $Z_{i}$.
Since each component of $H[Z_i]$ that contains an edge in $X_1 \cup \cdots \cup X_n$ contains exactly one of the cycles $C_1, \ldots, C_i$, $Z_i \cup e_{i+1}$ contains a unique loose handcuff $H_{i+1}$ containing $C_{i+1}$. 
Since $H_{i+1}$ is a circuit of $F(H,\Bb)$ it is dependent in $M$, so $M$ has a circuit $W_{i+1} \subseteq H_{i+1}$. 
The circuit $W_{i+1}$ contains an edge $f_{i+1}$ not contained in $X$ (else $W_{i+1} \subset X$), and so $f_{i+1}$ is in the path linking the two cycles of the loose handcuff $H_{i+1}$.
Since $W_{i+1}$ is the unique circuit contained in $Z_i \cup e_{i+1}$ and deleting $f_{i+1}$ destroys $W_{i+1}$, $Z_i \cup e_{i+1} - f_{i+1}$ is independent.
Set $Z_{i+1} = Z_i \cup e_{i+1} - f_{i+1}$.
Now $H[Z_{i+1}]$ has $i+1$ unicyclic components, containing cycles $C_1, \ldots, C_{i+1}$, respectively.
For each $j \in \{i+2, \ldots, n\}$ the edge set $C_j - e_j$ is contained in one of these components.

This process ends after $n-1$ steps with an independent set $Z_n$.
We started with a set of edges containing $X \del \{e_2, \ldots, e_n\}$ (the set $Z_1$), added each of $e_2, \ldots, e_n$ to this set, and removed only edges not in $X$. 
So we have reached the contradiction that 
the independent set $Z_n$ contains the circuit $X$. 
\end{proof}

{\color{black}
Lemma \ref{lem:a_circuit_has_at_most_two_components} says that axiom 
% (\ref{BGM(3)}) 
\BGthree\ 
for biased-graphic matroids almost achieves axiom 
% (\ref{(4)}) 
\QGfour\ 
for quasi-graphic matroids as a consequence: it falls short just in that rather than every circuit inducing a subgraph of at most two components, we are guarenteed only that each circuit induces a subgraph of Betti number at most two. 
We will subsequently see that if $X$ is a circuit of a biased-graphic matroid $M$ with graph $G$ satisfying axiom 
\BGthree, 
and $G[X]$ has more than two connected components, then $M$ is of a particular easily characterised form (Theorem \ref{broken_handcuff_matroids}). 
}

\subsection{Bracelet functions and tripartions}

We now show that our descriptions of quasi-graphic matroids using bracelet functions and proper tripartitions in fact describe biased-graphic matroids. 

\begin{theorem} \label{P:BraceletChoice}
Let $M$ be a biased-graphic matroid with graph $G$. 
If $G$ is 2-connected, then the circuits of $M$ are precisely those
edge sets given by $\Cc(G,\Bb,\chi)$ for some proper bracelet function $\chi$ on the bracelet graph of $\GB$. 
\end{theorem}

\begin{proof}
Let $\chi$ be the bracelet function defined according to the independence or dependence of each bracelet in $M$; that is, for each bracelet $B$ of $G$, $\chi(B) = \mathsf{independent}$ if and only if $B \in \Ii(M)$. 

($\Cc(M)\subseteq\Cc(G,\Bb,\chi)$.)
Let $X \in \Cc(M)$. 
If $X$ forms a cycle in $G$ then $X \in \Bb \in \Cc(G,\Bb,\chi)$, so assume $X$ is not a cycle. 
Certainly then neither does $X$ contain a balanced cycle. 
By Lemma \ref{lem:a_circuit_has_at_most_two_components}, $\beta(X)\leq 2$. 
If $\beta(X)\leq1$ then $X$ would be independent in $L(G,\mc B)$ and so independent in $M$; thus $\beta(X)=2$. 
Therefore $G[X]$ contains a theta, a handcuff, or a bracelet. 
Suppose $T \subseteq X$ is the edge set of a theta or tight handcuff in $G$. 
Then $T$ is dependent in $F\GB$ and so dependent in $M$, while every proper subset of $T$ is independent in $L\GB$ and so independent in $M$. 
Thus $X=T$, and $X \in \Cc(G,\Bb,\chi)$. 
Now suppose $X$ contains a bracelet $C_1 \cup C_2$. 
Every proper subset of $C_1 \cup C_2$ is independent in $L\GB$ and so independent in $M$. 
Thus either $X = C_1 \cup C_2$ and $\chi(C_1 \cup C_2) = \mathsf{dependent}$ or $X$ properly contains $C_1 \cup C_2$ and $\chi(C_1 \cup C_2) = \mathsf{independent}$. 
In the first case, $X \in \Cc(G,\Bb,\chi)$. 
So assume the second case holds, and suppose for a contradiction that $X$ is not a loose handcuff. 
Since every loose handcuff in $G$ is dependent in $F\GB$, and so dependent in $M$, $X$ does not contain a loose handcuff. 
Therefore there is no path in $G[X]$ connecting $C_1$ and $C_2$. 
Let $e_1 \in C_1$ and $e_2 \in C_2$. 
Then $X \bs \{e_1,e_2\}$ forms a forest in $G$ with $C_1-e_1$ and $C_2-e_2$ in different components. 
Since $X$ properly contains $C_1\cup C_2$, there is an element $x \in  X-(C_1\cup C_2)$ such that $x$ has an end $v$ of degree 1 in $G[X]$. 
Because $G$ is 2-connected, $G-v$ is connected. 
Hence there is a set of edges $W$ such that $\br{X \cup W} - \{e_1, e_2, x\}$ is a tree spanning $G-v$ and $(X\cup W)-\{e_1,e_2\}$ is the edge set of a spanning tree of $G$ in which $v$ is a leaf.
Now $(X\cup W)-x$ contains a unique loose handcuff consisting of $C_1\cup C_2$ along with the path $P$ linking $C_1$ and $C_2$ in $X \cup W$. 
Since $C_1 \cup C_2 \cup P$ is dependent in $F\GB$, $C_1 \cup C_2 \cup P$ is dependent in $M$. 
Thus there is a circuit $X' \neq X$ of $M$ contained in $(X \cup P)- x$. 
By the circuit elimination axiom there is a circuit $X''\subseteq (X\cup X')-e_1\subseteq (X\cup P)-e_1$. 
But $\beta((X \cup P)-e_1)=1$, so $(X \cup P)-e_1$ is independent in $L\GB$. 
Thus $(X \cup P)-e_1$ is independent in $M$, a contradiction. 

($\Cc(G,\Bb,\chi) \subseteq \Cc(M)$.) Every circuit $X$ of $M(G, \Bb, \chi)$ is dependent in $M$:  If $X$ is the edge set of a balanced cycle, theta subgraph with no cycle in $\Bb$, or handcuff of $(G, \Bb)$ then $X$ is already dependent in $F(G, \Bb)$, and if $X$ is a bracelet with $\chi(X) = \mathsf{dependent}$ then this is true by the definition of $\chi$. Thus $X$ must include a circuit $X'$ of $M$, which as we have just shown is then a circuit of $M(G, \Bb, \chi)$. Since no circuit of $M(G, \Bb, \chi)$ is a proper subset of any other, we have $X' = X$ and therefore $X \in \Cc(M)$.

This completes the proof that $\Cc(G,\Bb,\chi) = \Cc(M)$.
By Theorem \ref{thm:BraceletFunctionPropriety}, $\chi$ is proper. 
\end{proof}

For convenience sake we record the following immediate consequence of Theorems 
\ref{fixed_2_thm:equivalences}
and \ref{P:BraceletChoice}. 

\begin{theorem} \label{thm:BG_with_2conngraph_is_QG}
{Let $M$ be a biased-graphic matroid with a 2-connected graph. Then $M$ is quasi-graphic. }
\end{theorem}

\begin{proof}
Let $G$ be a 2-connected graph for $M$. 
By Theorem \ref{P:BraceletChoice}, $M=M(G,\Bb,\chi)$ for some proper bracelet function $\chi$, so by Theorem \ref{fixed_2_thm:equivalences}, $M$ is quasi-graphic. 
\end{proof}

{\color{black}
We will need the rank function of a connected biased-graphic matroid in terms of its credentials. 

\begin{proposition} \label{BGM_rank_when_graph_connected}
Let $M$ be a biased-graphic matroid with credentials $(G,\Bb,H)$, 
and let $X \subseteq E(M)$ where $H[X]$ is connected. The rank of $X$ in $M$ is $|Z|$, where $Z$ is a maximal unicyclic subgraph of $H[X]$ containing no balanced cycle. 
\end{proposition} 

\begin{proof}
The set $Z$ is independent in $L(G,\Bb)$, so independent in $M$. 
For each element $e \in X$, $H[Z \cup e]$ contains a 
balanced cycle, theta or handcuff and so is dependent in $F(H,\Bb)$, and so dependent in $M$. 
\end{proof}

Let $c(G)$ denote the number of components of the graph $G$. 
}

{\color{black}
\begin{proposition} \label{lem_conn_bias_gr_rank}
Let $M$ be a connected biased-graphic matroid with credentials $(G,\Bb,H)$. 
Assume that not every cycle of $G$ is in $\Bb$. 
Then the rank of $M$ is $|V(G)|-c(G)+1 = |V(H)|$. 
\end{proposition}

\begin{proof}
Let $C$ be an unbalanced cycle of $G$ and let $Z$ be a maximal unicyclic set of edges of $G$ containing $C$. 
Then $Z$ is a maximal unicyclic set of edges of $H$ containing no balanced cycle. 
As $H$ is connected, by Proposition \ref{BGM_rank_when_graph_connected} 
the rank of $M$ is equal to $|Z|$, which is equal to $|V(G)|-c(G)+1 = |V(H)|$. 
\end{proof}
}

{\color{black}
\begin{proposition} \label{lem:M3conn_implies_G2conn}
Let $M$ be a 3-connected biased-graphic matroid with graph $G$. 
Then $G$ is 2-connected.
\end{proposition}
}

\begin{proof} 
{\color{black}
Suppose to the contrary that $G$ is not 2-connected. 
Let $(G,\Bb,H)$ be credentials for $M$. 
Let $G_1, \ldots, G_n$ be the blocks of $G$. 
Then $G_1, \ldots, G_n$ are the blocks of $H$. 
Each block is unbalanced, else $M$ would not be connected. 
Consider the block-cut-point tree of $H$. 
Let $v$ be a vertex of block $G_1$ incident to a neighbouring block in the block-cut-point tree of $H$. 
Let $(A,B)$ be a separation of $H$ for which $V(A) \cap V(B) = \{v\}$. 
Then each of $H[A]$ and $H[B]$ are connected. 
Let $S \subseteq A$, and $T \subseteq B$, respectively, be maximal unicyclic subgraphs of $A$, resp.\ $B$, containing no balanced cycle. 
Then $|S| = |V(A)|$, $|T| = |V(B)|$. 
The separation $(A,B)$ is a separation of $M$, and by Propositions \ref{BGM_rank_when_graph_connected} and  \ref{lem_conn_bias_gr_rank}, 
\begin{align*}
r(A) + r(B) - r(M) 
&= |S| + |T| - |V(H)| \\ 
&= |V(A)| + |V(B)| - V(H) = 1 
\end{align*}
so $(A,B)$ is a 2-separation of $M$, contrary to the fact that $M$ is 3-connected. 
}
\end{proof}

{\color{black}
Proposition \ref{lem:M3conn_implies_G2conn} implies all 3-connected biased-graphic matroids are quasi-graphic: 

\begin{theorem} \label{thm:3conn_BG_is_QG}
Let $M$ be a 3-connected biased-graphic matroid. 
Then $M$ is quasi-graphic.
\end{theorem}

\begin{proof}
Let $G$ be a graph for $M$. 
By Proposition \ref{lem:M3conn_implies_G2conn}, $G$ is 2-connected. 
Thus by Theorems \ref{fixed_thm:equivalences} and \ref{P:BraceletChoice}, $M$ is quasi-graphic. 
\end{proof}
}

{\color{black}
By Theorems \ref{thm:allquasigraphicarebiasedgraphic} and \ref{thm:3conn_BG_is_QG}, 
the classes of 3-connected biased-graphic matroids and 3-connected quasi-graphic matroids coincide. 
In our next subsection (Subsection \ref{sec_link_and_loop_sums}), we shall see precisely how to construct those biased graphic matroids that are not quasi-graphic.
}

{\color{black}If $M(G,\Bb,\Ll,\Ff)$ has $\Ll$ degenerate, so $M(G,\Bb,\Ll,\Ff) = F(G,\Bb)$, then we say $(G,\Bb,\Ll,\Ff)$ is a \emph{frame representation} for $M$. 
Similarly, if $M(G,\Bb,\Ll,\Ff)$ has $\Ff$ degenerate, so $M(G,\Bb,\Ll,\Ff) = L(G,\Bb)$, then we say $(G,\Bb,\Ll,\Ff)$ is a \emph{lifted-graphic representation} for $M$. 
If $M$ is biased-graphic with a 2-connected graph $G$, then by Theorem 
\ref{thm:BG_with_2conngraph_is_QG},
$M=M(G,\Bb,\Ll,\Ff)$ for some proper tripartition of the cycles of $G$. 
If there is such a tripartition for $G$ with $\Ll$ (respectively, $\Ff$) degenerate, then we simply say that $G$ is a \emph{frame representation} (resp.\ \emph{lifted-graphic representation}) for $M$.
}
The following is the analogue of Theorem \ref{thm:quasigraphic_with_a_loop} for biased-graphic matroids. 

\begin{theorem} \label{thm:Gtwoconnectedwithloop}
Let $M$ be a biased-graphic matroid and let $G$ be a graph for $M$. If $G$ is 2-connected and has an unbalanced loop, then 
{\color{black}$G$ is either a lifted-graphic representation or a frame representation for $M$, so}
$M$ is either lifted-graphic or frame.
\end{theorem}

\begin{proof}
By Theorem \ref{thm:BG_with_2conngraph_is_QG} 
the circuits of $M$ are given by $\Cc(G,\Bb,\Ll,\Ff)$ for some proper tripartition of the cycles of $G$. 
Since every cycle in $\Ll$ meets every cycle in $\Ff$, one of $\Ll$ or $\Ff$ is degenerate. 
\end{proof}

{\color{black}
For the sake of completeness, we provide the general rank function for a quasi-graphic matroid in terms of its credentials. 
While $c(G)$ denotes the number of components of the graph $G$, $b(G)$ denotes the number of balanced components of $G$. 

\begin{proposition} \label{lem_general_rank_function}
Let $M$ be a biased-graphic matroid with credentials $(G,\Bb,H)$. 
Let 
$(G_1, \ldots, G_m)$ be a partition of the set of connected components of $G$ into subgraphs for which distinct connected components of $G$ are contained in the same subgraph $G_j$ only if there is a component of $M$ with non-empty intersection with each of them. 
Then the rank of $M$ is 
\[ |V(H)| - b(H) 
= |V(G)| - c(G) + u \] 
where 
$u$ is the number of subgraphs in the partition $(G_1, \ldots, G_m)$ that are unbalanced. 
\end{proposition}

\begin{proof}
Let $Z$ be a maximal set of elements of $M$ subject to $Z$ inducing in each connected component $H_i$ of $H$ a subgraph containing at most one cycle and no balanced cycle. 
Then $|Z| = |V(H)|-b(H)$. 
Let $L$ be the direct sum of the matroids $L(G_j,\Bb_j)$, $j \in \{1,\ldots, m\}$, where $\Bb_j$ is the set of cycles of $G_j$ that are circuits of $M$. 
For each $j$, let $Z_j = Z \cap E(G_j)$. 
Then $Z_j$ is a basis of $L(G_j,\Bb_j)$. 
Since $Z$ is a basis of $L$, and 
\[ \Ii(L(G,\Bb)) \subseteq \Ii(L) \subseteq \Ii(M) \] 
$Z$ is independent in $M$. 
For each element $e \in E(M)-Z$, $Z \cup e$ is dependent in $F(H,\Bb)$ and so dependent in $M$. 
Thus $Z$ is a basis for $M$, and so the rank of $M$ is 
\begin{align*}
|Z| = |V(H)| - b(H) 
&= \sum_{i=1}^m \br{|V(G_i)| - c(G_i) + l(G_i)} \\
% &= |V(G)| - c(G) + u(G)
&= |V(G)| - c(G) + u
\qedhere 
\end{align*}
\end{proof}
}

\subsection{Link-sums and loop-sums} \label{sec_link_and_loop_sums}

None of the classes of frame, lifted-graphic, quasi-graphic, nor biased-graphic matroids are closed under 2-sums. 
However, each is closed under 2-sums with a graphic matroid. 
Moreover, there is a natural 2-sum of biased graphs that corresponds to a 2-sum of matroids that, for the classes of lifted-graphic, frame, and biased-graphic matroids, does always yield a matroid that remains in the class containing the summands. 

Let $M_1$ and $M_2$ be two biased-graphic matroids. 
Let $G_1$ be a graph for $M_1$ and let $G_2$ be a graph for $M_2$, and assume $E(M_1) \cap E(M_2) = \{e\}$. 
If $e$ is a link in $G_1$ and every cycle of $G_2$ is a circuit of $M_2$ 
(that is, 
$M_2$ is the cycle matroid of $G$) 
then the 2-sum of $M_1$ and $M_2$ on basepoint $e$ may be realised in the graphs as follows.
Suppose $u_1,v_1$ are the endpoints of $e$ in $G_1$ and $u_2, v_2$ are the endpoints of $e$ in $G_2$.
The \emph{link-sum} of $G_1$ and $G_2$ \emph{on} $e$, denoted $G_1 \oplus_2^e G_2$, is the graph obtained by identifying $u_1$ with $u_2$ and identifying $v_1$ with $v_2$, then deleting $e$. 
It is straightforward to check that $G_1 \oplus_2^e G_2$ is a graph for the matroid $M_1 \oplus_2^e M_2$ obtained as the 2-sum of $M_1$ and $M_2$ on basepoint $e$.

Now suppose that $e$ is an unbalanced loop in each of $G_1$ and $G_2$.
The \emph{loop-sum} of $G_1$ and $G_2$ \emph{on} $e$, denoted $G_1 \oplus_2^e G_2$, is obtained by identifying the end of $e$ in each of $G_1$ and $G_2$, then deleting $e$. 
Again it is straightforward to check that $G_1 \oplus_2^e G_2$ is a graph for the 2-sum of $M_1$ and $M_2$ taken on basepoint $e$. 
It is easy to see that if $G_1$ and $G_2$ represent $M_1$ and $M_2$ as frame matroids, then the loop-sum of $G_1$ and $G_2$ is a frame representation of the 2-sum of $M_1$ and $M_2$ on $e$. 
Similarly, if $G_1$ and $G_2$ represent $M_1$ and $M_2$ as lifted-graphic matroids, then the loop-sum of $G_1$ and $G_2$ is a lifted-graphic representation of the 2-sum of $M_1$ and $M_2$ on $e$. 
Thus each of the classes of frame and lifted-graphic matroids are closed under loop-sums. 
The class of quasi-graphic matroids, however, is not closed under loop-sums. 

\begin{proposition} \label{lem:twosumframeandliftnotqg}
Let $M$ be a frame matroid that is not lifted-graphic and let $N$ be a lifted-graphic matroid that is not frame. Assume $E(M) \cap E(N) = \{e\}$. Let $G$ be a graph for $M$ and let $H$ be a graph for $N$, and assume that $e$ is an unbalanced loop in each of $G$ and $H$.
Then $M \oplus_2^e N$ is biased-graphic but not quasi-graphic. 
\end{proposition}

\begin{proof} 
It is straightforward to check that the loop-sum $G \oplus_2^e H$ is a graph for $M \oplus_2^e N$, so $M \oplus_2^e N$ is biased-graphic.
Let $C$ be a circuit of $M$ containing $e$ that forms a loose handcuff in $G$.
Let $C'$ be a circuit of $N$ containing $e$ that forms a bracelet in $H$.
Then $(C-e) \cup (C'-e)$ is a circuit of $M \oplus_2^e N$ that does not have the form of any of the circuit-subgraphs of Lemma \ref{lem:GGWquasicircuits}, so $M \oplus_2^e N$ is not quasi-graphic. 
\end{proof}

{
Let $G$ be a graph for a connected biased-graphic matroid $M$, and suppose $G$ is not 2-connected. 
Since $M$ is connected, 
% \st{$G$ is connected}
{\color{black}$M$ has credentials $(G,\Bb,H)$ where $H$ is connected. 
As $G$ is not 2-connected, neither is $H$.}
Let $v$ be a cut vertex of 
{\color{black}$H$}, and let $(A,B)$ be a partition of 
{\color{black}$E(H)$} so that $V(A) \cap V(B) = \{v\}$. 
{\color{black}The argument in the proof of Proposition \ref{lem:M3conn_implies_G2conn} shows that $(A,B)$ is a 2-separation of $M$.}
Let $e$ be an element not in $E(M)$. 
Let $H_A$ be the graph obtained from $H[A]$ by adding a loop incident to $v$ labelled $e$, and let $H_B$ be the graph obtained from $H[B]$ obtained by adding a loop labelled $e$ incident to $v$. 
Then $H_A$ is a graph for the biased-graphic matroid $M_A$ on $A \cup \{e\}$ whose circuits are the circuits of $M$ that are contained in $A$ together with elements of the set 
\[ 
\{\br{C-B} \cup e : C \in \Cc(M), C \cap B \neq \emptyset \}.
\]
The graph $H_B$ is a graph for the analogously defined matroid $M_B$ on $B \cup \{e\}$. 
Evidently, $H=H_A \oplus_2^e H_B$ and $M = M_A \oplus_2^e M_B$. 
More generally, let $K$ be a block of $H$, and let $v_1, \ldots, v_k$ be the cut vertices of $H$ in $K$. 
For each cut vertex $v_i$ there is a 2-separation $(A_i, B_i)$ of $M$, where $E(K) \subseteq A_i$.  
Let $M_K$ be the matroid in the 2-sum decomposition of $M$ corresponding to this star of 2-separations. 
The matroid $M_K$ has an extra element $e_i$ for each cut vertex $v_i$ ($i \in \{1, \ldots, k\}$) of $H$ in $K$. 
Let $K^+$ be the graph obtained from $K$ by adding, for each $i \in \{1, \ldots, k\}$, $e_i$ as a loop incident to $v_i$. 
Then $K^+$ is a graph for the biased-graphic matroid $M_K$ on $E(K) \cup \{e_1, \ldots, e_k\}$ whose circuits are 
the elements of the set 
\[ \{ \br{C \cap E(K)} \cup \{e_i : C \cap B_i \neq \emptyset \} : C \in \Cc(M) \}.\]
Moreover, in this way each block $K$ of $H$ corresponds to a matroid $M_K$ in the 2-sum decomposition of $M$, and if $K$ and $J$ are two blocks of $H$ with $V(K) \cap V(J) = \{v_i\}$, then the loop sum $K^+ \oplus_2^{e_i} J^+$ yields a graph for the matroid $M_K \oplus_2^{e_i} M_J$. 
}

Equipped with this simple decomposition tool, we may now prove the following. 

\begin{proposition} \label{thm:blocks_of_G_are_frame_or_lift}
{\color{black}
Let $M$ be a biased-graphic matroid, 
with credentials $(G,\Bb,H)$. 
If $H'$ is a connected component of $H$ that is not 2-connected and $K, K'$ are two blocks of $H'$ sharing a vertex, both containing elements that are contained in a common component of $M$, then 
the restriction of $H'$ to each of $K$ and $K'$ is a lifted-graphic or frame representation for the restriction of $M$ to the elements of that block.}
\end{proposition}

\begin{proof} 
{\color{black}Each of $K$, $K'$}
is either 2-connected or has at most two vertices. 
{\color{black}Every graph for a matroid that has at most two vertices is either a lifted-graphic or a frame representation for the matroid,}
so assume $|V(K)| \geq 3$.
{\color{black}Let $v$ be the vertex in $V(K) \cap V(K')$.} 
The matroid $M_K$ is biased-graphic with graph $K^+$ as defined above; $K^+$ has a loop 
{\color{black}incident to $v$.} 
Thus by Theorem \ref{thm:Gtwoconnectedwithloop}, 
{\color{black}the representation $K^+$ for}
$M_K$ is either lifted-graphic or frame
{\color{black}Similarly, the representation $K'^+$ for $M_{K'}$ is lifted-graphic or frame.}
\end{proof}

{\color{black}
Theorem \ref{thm:blocks_of_G_are_frame_or_lift} has the following immediate corollary. 

\begin{corollary} \label{frame_or_lift_reps}
Let $M$ be a connected biased graphic matroid that is not 3-connected and let $G$ be a connected graph for $M$. 
Assume $G$ has more than one block, and let $B$ be a block of $G$. 
The representation of the restriction of $M$ to the elements of $B$ given by $G[B]$ is either a frame representation or a lifted-graphic representation. 
\end{corollary} 
}

{\color{black}
\subsection{The class of biased-graphic matroids is closed under minors} \label{sec_BGM_minors}

We can now show that the class of biased-graphic matroids is closed under minors. 
Just as for each of the classes of 
graphic, 
lifted-graphic, 
frame, and 
quasi-graphic 
matroids, corresponding minor operations in graph representations yield graph representations of minors. 
We observe the following widely-accepted conventions to ensure that this is so, and add one more. 
\begin{itemize} 
\item When $M$ is the cycle matroid of the graph $G$, or $M$ is lifted graphic, or frame, or quasi-graphic, with graph $G$, and $e$ is a loop of $M$ (and so a balanced loop in $G$), we contract $e$ from $G$ by deleting $e$. 
\item When $M=L(G,\Bb)$ is lifted-graphic represented by the biased graph $(G,\Bb)$ and $e$ is an unbalanced loop in $G$ --- equivalently, if $M=M(G,\Bb,\Ll,\Ff)$ is quasi-graphic with framework $G$, and $e$ is a loop in $\Ll$ --- we contract $e$ from $G$ by deleting $e$ and declaring all cycles balanced. 
\item When $M=F(G,\Bb)$ is a frame matroid represented by $(G,\Bb)$ and $e$ is an unbalanced loop in $G$ --- equivalently, $M=M(G,\Bb,\Ll,\Ff)$ and $e$ is a loop in $\Ff$ --- we contract $e$ from $G$ by deleting $e$, declaring all unbalanced loops adjacent to $e$ to be balanced, and replacing each non-loop edge sharing an endpoint with $e$ with an unbalanced loop incident to its other endpoint (adding each such loop to the set $\Ff$). 
\end{itemize} 
Just as a loop in a graph representation for a matroid may or may not represent a loop of the matroid, 
a bridge of a graph may or may not represent a coloop of the matroid, according to whether or not the bridge is or is not contained in loose handcuffs induced by a circuit of the matroid. 
Just as when contracting a loop in a graph representation of a matroid we apply differing contraction operations depending on whether the loop is a loop of the matroid or not (and on whether the graph is a lifted-graphic or frame representation), 
we require differing deletion operations depending upon whether or not a bridge represents a coloop of the matroid. 
Thus to the above list, we add the following convention. 
\begin{itemize} 
\item When $M$ is a biased-graphic matroid with credentials $(G,\Bb,H)$, and $e$ is a coloop of $M$, we delete $e$ from $H$ by contracting $e$. 
\end{itemize} 
The following theorem justifies this convention. 
It says that with these conventions, if $M$ is a biased-graphic matroid and $G$ is graph for $M$ 
satisfying conditions 
\BGone, \BGtwo, and \BGthree\ 
of the definition, 
and $e$ is an element of $M$, 
then graphs for $M \del e$ and $M/e$ 
satisfying conditions 
\BGone, \BGtwo, and \BGthree\ 
are obtained by applying the corresponding minor operation in $G$. 
Given a biased graph $(G,\Bb)$ and an edge $e$ of $G$, let $\Bb_{G/e}$ be the collection of cycles of $G/e$ defined by $\Bb_{G/e} = \{C : C \in \Bb$ or $C \cup e \in \Bb\}$ provided $(G,\Bb)$ is not a lifted-graphic representation and $e$ is not an unbalanced loop; 
if $(G,\Bb)$ is a lifted-graphic representation and $e$ is an unbalanced loop, then define $\Bb_{G/e}$ to be the collection of all cycles of $G$. 
Let $\Bb_{G \del e}$ denote the collection of balanced cycles of $G$ that do not contain $e$. 

\begin{theorem} \label{BGM_minor-closed_1}
Let $M$ be a biased-graphic matroid with credentials $(G,\Bb,H)$, and let $e \in E(M)$. 
Then $(G/e, \Bb_{G/e}, H/e)$ are credentials for $M/e$, and 
provided $e$ is not a coloop of $M$, 
$(G \del e, \Bb_{G \del e}, H \del e)$ are credentials for $M \del e$. 
If $e$ is a coloop of $M$, then $(G \del e, \Bb_{G \del e}, H/e)$ are credentials for $M \del e$. 
\end{theorem} 

\begin{proof}
Let $(G,\Bb,H)$ be credentials for $M$, and let $e \in E(M)$. 
First consider contraction. 
The collection $\Bb_{G/e}$
is equal to the set $\{ C : C$ is a circuit of $M/e$ and a cycle of $G/e\}$, and satisfies the theta property because $\Bb$ does (a theta-property-violating set of cycles in $\Bb_{G/e}$ would imply the existence of a theta-property-violating set in $\Bb$). 
We claim that $(G/e, \Bb_{G/e}, H/e)$ are credentials for $M/e$. 
As $L(H,\Bb)/e = L(H/e,\Bb_{G/e})$ and $F(H,\Bb)/e = F(H/e,\Bb_{G/e})$, the statement 
$$\Ii(L\br{H/e,\Bb_{G/e}}) \subseteq \Ii\br{M/e} \subseteq \Ii\br{F(H/e,\Bb_{G/e})}$$
follows from the statement 
$\Ii(L\br{H,\Bb}) \subseteq \Ii\br{M} \subseteq \Ii\br{F(H,\Bb)}$. 
The contraction of $e$ in $G$ and $H$ does not produce an additional component of either graph, so the fact that $H$ is obtained by vertex identification from $G$ and that every component of $M$ is contained in a component of $H$ implies that $H/e$ is obtained by vertex identification from $G/e$, and that every component of $M/e$ is contained in a component of $H/e$. 
Thus 
\BGone, \BGtwo, and \BGthree\
hold for $M/e$ and $H/e$. 

Now consider deletion. 
First assume $e$ is not a coloop of $M$. 
We show that $(G \del e, \Bb_{G\del e}, H \del e)$ are credentials for $M \del e$. 
The set $\Bb_{G\del e}$ satisfies the theta property, since $\Bb$ does, and is equal to the set $\{C : C$ is a circuit of $M \del e$ and a cycle of $H \del e\}$. 
Since $L(H,\Bb) \del e = L(H \del e, \Bb_{G \del e} )$, and 
$F(H,\Bb) \del e = F(H \del e, \Bb_{G\del e})$, 
the statement 
\[\Ii(L(H \del e,\Bb_{G\del e}) \subseteq \Ii(M \del e) \subseteq \Ii(F(H \del e,\Bb_{G\del e}))\]
follows from the fact that $\Ii(L(H,\Bb) \subseteq \Ii(M) \subseteq \Ii(F(H,\Bb))$. 

Note that $e$ is a bridge of $G$ if and only if $e$ is a bridge of $H$. 
If $e$ is not a bridge, 
then $H \del e$ has no new component, so 
each component of $M \del e$ is contained in a component of $H \del e$. 
Thus 
\BGone, \BGtwo, and \BGthree\
hold for $M \del e$ and $H \del e$. 

But suppose $e$ is a bridge. 
Let $H_1$, $H_2$ be the two new components of $H$ resulting from the deletion of $e$. 
As $e$ is not a coloop there is a circuit of $M$ containing $e$ along with elements of $H_1$ and $H_2$. 
This circuit must be a loose handcuff in $H$. 
Let $u,v$ be the endpoints of $e$, say with $u \in V(H_1)$ and $v \in V(H_2)$. 
By the argument in the proof of Proposition \ref{lem:M3conn_implies_G2conn} and the discussion of 2-sums and loop-sums in Subsection \ref{sec_link_and_loop_sums}, 
there are disjoint subsets $A, B$ of $E(M) - e$ such that $A \cup B \cup e = E(M)$ and each of $(A, B \cup e)$ and $(A \cup e, B)$ are 2-separations of $M$, where $V(A) \cap V(B \cup e) = u$ and $V(A \cup e) \cap V(B) = v$. 
Thus $H$ may be expressed as a pair of loop-sums $H_A \oplus_2 H_{C_e} \oplus_2 H_B$, for appropriately defined biased graphs $H_A$, $H_{C_e}$, $H_B$ (where $H_{C_e}$ is a loose handcuff consisting of $e$ as a link between a pair of loops), 
and $M$ may be expressed as a pair of 2-sums $M_A^+ \oplus_2 C_e \oplus_2 M_B^+$, where $C_e$ is a 3-circuit containing $e$ along with two basepoints for the 2-sums. 
Thus each of $A$ and $B$ are contained in separate components of $M \del e$, 
so $(G \del e, \Bb_{G \del e}, H \del e)$ are credentials for $M \del e$. 

Now assume that $e$ is a coloop of $M$. 
Then $M \del e = M/e$, so by the first part of this proof, $(G \del e, \Bb_{G \del e}, H/e)$ are credentials for $M \del e$. 
\end{proof}

By Theorem \ref{BGM_minor-closed_1}, the class of biased-graphic matroids is closed under minors, and corresponding minor operations in graph representations yield graph representations of minors. 
}

\subsection{Quasi-graphic and biased-graphic matroids} 

We now investigate the relationship between the classes of biased-graphic matroids and quasi-graphic matroids,
and characterise those biased-graphic matroids that are not quasi-graphic. 
It turns out that for matroids that are 
{\color{black}connected but} 
not 3-connected, the class of quasi-graphic matroids is nothing more than the union of the classes of lifted-graphic and frame matroids. 
While the class of biased-graphic matroids that are not 3-connected is larger than this, we will see that 
the class admits a straightforward characterisation in terms of lifted-graphic and frame summands in a 2-sum decomposition. 

{\color{black}We have already seen 
(Theorem \ref{thm:BG_with_2conngraph_is_QG}) 
that biased-graphic matroids with 2-connected graphs are quasi-graphic.}
We may now characterise the relationship between the classes quasi-graphic matroids and biased-graphic matroids. 

\begin{theorem} \label{thm:struc_of_qgraph}
A connected matroid $M$ is quasi-graphic if and only if $M$ is a frame matroid, a lifted-graphic matroid, or a biased-graphic matroid with a 2-connected graph. 
\end{theorem}

\begin{proof}
By Theorem \ref{thm:allquasigraphicarebiasedgraphic} all quasi-graphic matroids are biased-graphic. 
By Theorem \ref{thm:BG_with_2conngraph_is_QG} all biased-graphic matroids with a 2-connected graph are quasi-graphic. 
Thus we just need show that if $M$ is a connected biased-graphic matroid with a graph that is not 2-connected, then $M$ is quasi-graphic only if $M$ is frame or lifted-graphic. 
So let $M$ be a connected biased-graphic matroid with a graph $G$ that is not 2-connected, with proper tripartition $(\Bb,\Ll,\Ff)$. 
By 
{\color{black}Corollary \ref{frame_or_lift_reps} and the discussion prior}, 
each block $H$ of $G$ has a corresponding matroid $M_H$ in a 2-sum decomposition of $M$ that is either lifted-graphic or frame, and the 2-sums corresponding to cut vertices between two blocks are loop-sums. 
Since the class of frame matroids (resp.\ lifted-graphic matroids) is closed under loop-sums, if the restriction of $M$ to each block is frame (resp.\ lifted-graphic) then $M$ is frame (resp.\ lifted-graphic). 
Otherwise, there is a pair of blocks $H$, $H'$ such that the 2-sum of the matroids $M_H$ and $M_{H'}$ occurs in a 2-sum decomposition of $M$, where $M_H$ is lifted-graphic but not frame and $M_{H'}$ is frame but not lifted-graphic. 
By Proposition \ref{lem:twosumframeandliftnotqg}, this 2-sum is not quasi-graphic. 
\end{proof}

We can also characterise quasi-graphic matroids within the class of biased-graphic matroids in terms of matroid connectivity. 

\begin{theorem} 
Let $\Qq$ be the class of matroids consisting of the following sets of matroids, closed under direct sums: 
\begin{enumerate} 
\item 3-connected biased-graphic matroids, and 
\item biased-graphic matroids for which a 2-sum decomposition satisfies one of the following: 
\begin{enumerate}
\item there is exactly one 3-connected biased-graphic matroid while all remaining summands are graphic,  
\item every summand is lifted-graphic, or 
\item every summand is frame.  
\end{enumerate} 
\end{enumerate} 
Then $\Qq$ is the class of quasi-graphic matroids. 
\end{theorem}

\begin{proof}
By Theorems \ref{thm:allquasigraphicarebiasedgraphic} and \ref{thm:3conn_BG_is_QG} the classes of 3-connected quasi-graphic and biased-graphic matroids coincide. 

Let $M$ be a connected quasi-graphic matroid that is not 3-connected, and let $G$ be a graph for $M$. 
Suppose first that $G$ is 2-connected. 
Then either $M$ is frame or lifted-graphic, or by Lemma \ref{lem:M2conn_G2conn} $M$ is obtained from a 3-connected biased-graphic matroid by taking 2-sums with graphic matroids. 
Now suppose $G$ is not 2-connected. 
Then by 
Corollary \ref{frame_or_lift_reps}
the restriction of $M$ to each of its summands in a 2-sum decomposition is either lifted-graphic or frame. 
By Proposition \ref{lem:twosumframeandliftnotqg}, such a matroid is quasi-graphic if and only if either every summand is lifted-graphic or every summand is frame. 
\end{proof}

Thus the class of biased-graphic matroids differs from that of quasi-graphic matroids only in that a biased-graphic matroid is permitted to consist of 2-sums of frame matroids and lifted-graphic matroids, while for such a 2-sum to remain in the class of quasi-graphic matroids, either all summands must be frame or all summands must be lifted-graphic. 

We provide a third characterisation of those biased-graphic matroids that are not quasi-graphic. 
This characterisation suggests that the two classes of quasi-graphic and biased-graphic matroids are perhaps even closer than the previous result would indicate. 
It also specifies even more precisely just how the two classes differ. 
Let $\GB$ be a biased graph, 
let $U \subseteq V(G)$, 
and let $((G_v,\Bb_v) : v \in U)$ be a family of biased graphs with pairwise disjoint edge sets, all disjoint from $E(G)$.
There is a biased-graphic matroid whose circuits are the sets of the following types:
\begin{enumerate} 
\item circuits of $F\GB$,
\item circuits of $L(G_v, \Bb_v)$ for some $v$, 
\item unions of an unbalanced cycle $C$ of $G$, an unbalanced cycle $C'$ of some $G_v$, and a $(v$-$C)$-path in $G$, 
\item unions of a pair of unbalanced cycles $C \subseteq G_v$ and $C' \subseteq G_w$ for some distinct $v,w$, and a $(v$-$w)$-path in $G$. 
\end{enumerate}
To see that this really is a matroid, observe that it is obtained via successive 2-sums of a frame matroid corresponding to $F(G', \Bb)$ with members of a collection of lifted-graphic matroids corresponding to the $L(G_v',\Bb_v)$, where $G'$ is obtained from $G$ by adding a loop at each vertex in $U$ and each $G_v'$ is obtained from the corresponding $G_v$ by adding a single loop. 
Let us call such a matroid a \emph{broken handcuff matroid}, and circuits of types (3) and (4) above its \emph{broken handcuffs}. 

\begin{theorem} \label{broken_handcuff_matroids}
Let $M$ be a connected biased-graphic matroid. 
Then $M$ is either quasi-graphic or a broken handcuff matroid. 
\end{theorem}

\begin{proof} 
Let $M$ be a connected biased-graphic matroid that is not quasi-graphic, and let $H$ be a connected graph for $M$. 
By 
Theorem \ref{P:BraceletChoice}, 
$H$ is not 2-connected. 
Consider a block $K$ of $H$. 
To each cut vertex of $H$ in $K$ there corresponds a 2-separation of $M$; 
let $M_K$ be the matroid in the 2-sum decomposition of $M$ corresponding to this star of 2-separations. 
This matroid has an extra element $e_v$ for each cut vertex $v$ of $H$ in $K$; 
the graph $K^+$ obtained from $K$ by adding $e_v$ as a loop incident to $v$ is a graph for $M_K$. 
And $M_K$ is either frame or lifted-graphic since $K^+$ is 2-connected. Call $K$ a {\em lift block} if $M_K$ is lifted-graphic, and a {\em frame block} otherwise.

If $M_K$ is lifted-graphic then it is unaltered by moving all the loops to be incident with the same vertex $v$ of $K$. In $H$ this corresponds to replacing each edge from a vertex $w$ outside $K$ to $K$ with an edge from $w$ to $v$. 
Carry out this operation for all lift blocks; the matroid $M$ is unaffected. 
After these operations, the union of all frame blocks is a connected subgraph $G$ of $H$. 
For each vertex $v$ of $G$, let $G_v$ be the union of all lift blocks of $H$ which are cut from $G$ by $v$. Then $M$ is given as above from $G$ and the $G_v$.
\end{proof}

{\color{black}
We close with a list of statements that together provide necessary and sufficient conditions for a given graph $G$ to be a framework for a given matroid $M$, each of which may be checked in time polynomial in $|E(M)|$. 
The operation inverse to vertex identification is \emph{vertex cleaving}. 

\begin{theorem} \label{Felix}
Let $M$ be a matroid and let $G$ be a graph. $G$ is a framework for $M$ if and only if $G$ satisfies \QGone, \QGtwo, and \QGthree, and $G$ is obtained by a sequence of vertex cleaving operations from a graph $G'$ satisfying \QGone, \QGtwo, \QGthree, and \BGthree. 
\end{theorem} 

\begin{proof}
Suppose $G$ is a framework for $M$. 
Then by definition $G$ satisfies \QGone, \QGtwo, \QGthree, and \QGfour. 
If $G$ does not satisfy \BGthree, let $N$ be a component of $M$ for which $E(N)$ meets two connected components $H_1, H_2$ of $G$. 
Then there is a circuit $C$ of $M$ meeting both $E(H_1)$ and $E(H_2)$. 
By \cite[Lemma 3.3]{JGT:JGT22177} and \QGfour, $C$ induces a pair of unbalanced cycles, one in each of $H_1$ and $H_2$. 
Thus by Lemma \ref{lem:rank}, 
$r(E(H_1 \cup H_2)) 
= |V(H_1)| + |V(H_2)| - 1$. 
Let $H'$ be a graph obtained from $G$ by identifying a pair of vertices $u_1 \in V(H_1)$ and $u_2 \in V(H_2)$ to a single vertex $u$, and let $H$ be the resulting component of $H'$ containing $u$. 
Then $|V(H)| = |V(H_1)| + |V(H_2)| - 1$. 
Thus $r(E(H)) = r(E(H_1 \cup H_2)) = |V(H_1)| + |V(H_2)| - 1 = |V(H)|$. 
Every other component of $H'$ is a component of $G$, and so satisfies \QGtwo. 
Thus $H'$ satisfies \QGtwo. 
It is immediate that $H'$ satisfies \QGone; to see that $H'$ satisfies \QGthree, observe that for each vertex $v \in V(H')$ that is not $u$, \QGthree\ holds in $H'$ because $v$ has precisely the same set of incident edges as in $G$. 
So suppose $\cl(E(H'-u))$ contains a non-loop element $e$ incident to $u$ in $H'$; without loss of generality, assume $e \in E(H_1)$. 
Then $e$ is incident to $u_1$ in $G$. 
But $E(H'-u)$ is contained in $E(G-u_1)$, so if $e \in \cl(H'-u)$ then $e \in \cl(G-u_1)$, 
contrary to the fact that $G$ satisfies \QGthree. 

Thus $H'$ satisfies \QGone, \QGtwo, and \QGthree. 
If $H'$ also satisfies \BGthree, put $G'=H'$; otherwise choose a component of $M$ meeting a pair of distinct components of $H'$, and identify a pair of vertices, one from each component. 
Repeating as long as there is a component of $M$ meeting distinct components of the resulting graph, in this way we eventually obtain a graph $G'$ satisfying \QGone, \QGtwo, \QGthree, and \BGthree, 
from which $G$ is obtained by a sequence of vertex cleaving operations. 

Conversely, let $G$ be a graph satisfying \QGone, \QGtwo, and \QGthree, obtained by a sequence of vertex cleaving operations from a graph $G'$ satisfying \QGone, \QGtwo, \QGthree, and \BGthree. 
We just need show that $G$ satisfies \QGfour. 
So suppose not, and let $C$ be a circuit that induces in $G$ a subgraph with more than two components. 
By Lemma 3.3 of \cite{JGT:JGT22177}, $G[C]$ is a collection of vertex-disjoint unbalanced cycles, so $\beta_G(C) \geq 3$. 
The graphs $G$ and $G'$ have exactly the same set of cycles; 
put $\Bb = \{C : C$ is a cycle of $G$ and a circuit of $M\}$. 
Lemma 3.3 of \cite{JGT:JGT22177} now implies that 
$\Ii(L(G',\Bb)) \subseteq \Ii(M) \subseteq \Ii(F(G',\Bb))$.
Since $G'$ satisfies \BGthree, 
$M$ is biased-graphic with graph $G'$. 
Thus by Lemma \ref{lem:a_circuit_has_at_most_two_components}, 
$\beta_{G'}(C) \leq 2$. 
But $\beta_{G'}(C) = \beta_G(C)$, so this is a contradiction. 
\end{proof}

Given a matroid $M$ and a graph $G$, one can check \QGone, \QGtwo, and \QGthree\ for $G$ in time polynomial in $|E(M)|$. 
We can also find the components of $M$ and the components of $G$ in polynomial time, and so check whether $G$ satisfies \BGthree\ in polynomial time. 
If $G$ satisfies \QGone, \QGtwo, and \QGthree, but not \BGthree, then the 
process of vertex identification described in the first part of the proof of Theorem \ref{Felix} can be carried out in time polynomial in $|E(M)|$. 
If $M$ is quasi-graphic, in this way we obtain a framework $G'$ for $M$ satisfying \BGthree. 
Otherwise, 
since $G$ satisfies \QGone, \QGtwo, and \QGthree, $G$ violates \QGfour. 
In this case, the graph $G'$ we so obtain will satisfy \BGthree\ but violate \QGtwo. 
}

\subsection*{Acknowledgement} 
Our thanks to Jim Geelen for pointing out the missing connectivity assumption in our original paper, and for subsequent conversations. 

\providecommand{\bysame}{\leavevmode\hbox to3em{\hrulefill}\thinspace}
\providecommand{\MR}{\relax\ifhmode\unskip\space\fi MR }
\providecommand{\MRhref}[2]{%
  \href{http://www.ams.org/mathscinet-getitem?mr=#1}{#2}
}
\providecommand{\href}[2]{#2}

\end{document}